\documentclass[a4paper,12pt,oneside]{amsart}
\usepackage{geometry}
\usepackage{epstopdf}
\usepackage{mathtools}
\usepackage[cmtip,all]{xy}

\usepackage{mathrsfs}
\usepackage{amscd}
\usepackage{amsmath}
\usepackage{amssymb}
\usepackage{amsthm}
\usepackage{float}
\usepackage[dvips]{graphicx}
\usepackage{color}
\usepackage{tikz} 
\usepackage{enumerate}

\usetikzlibrary{matrix,arrows}

\usepackage{array}
\usepackage{amscd}
\usepackage[all]{xy}

\DeclareFontFamily{U}{rsfs}{%
\skewchar\font127}
\DeclareFontShape{U}{rsfs}{m}{n}{%
<-6>rsfs5<6-8.5>rsfs7<8.5->rsfs10}{}
\DeclareSymbolFont{rsfs}{U}{rsfs}{m}{n}
\DeclareSymbolFontAlphabet
{\mathrsfs}{rsfs}
\DeclareRobustCommand*\rsfs{%
\@fontswitch\relax\mathrsfs}

\DeclareFontFamily{U}{rsfs}{%
\skewchar\font127}
\DeclareFontShape{U}{rsfs}{m}{n}{%
<-6>rsfs5<6-8.5>rsfs7<8.5->rsfs10}{}
\DeclareSymbolFont{rsfs}{U}{rsfs}{m}{n}
\DeclareSymbolFontAlphabet
{\mathrsfs}{rsfs}
\DeclareRobustCommand*\rsfs{%
\@fontswitch\relax\mathrsfs}

\theoremstyle{plain}
\newtheorem{theorem}{Theorem}
\newtheorem{thm}{Theorem}[section]
\newtheorem{prop}[thm]{Proposition}
\newtheorem{lem}[thm]{Lemma}

\newtheorem{defi}[thm]{Definition}
\newtheorem{rmk}[thm]{Remark}
\newtheorem{cor}[thm]{Corollary}

\newtheorem*{prop*}{Proposition}
\newtheorem*{notn}{Notation}
\newtheorem*{conv}{Convention}

\newtheorem{prop-defi}[thm]{Proposition-Definition}
\newtheorem{thm-defi}[thm]{Theorem-Definition}
\newtheorem{lem-defi}[thm]{Lemma-Definition}

\newcommand{\C}{\mathbb{C}}
\newcommand{\ZZ}{\mathbb{Z}}
\newcommand{\Z}{\mathcal{Z}}
\newcommand{\I}{\mathcal{I}}
\renewcommand{\div}{\operatorname{div}}
\renewcommand{\det}{\operatorname{det}}
\newcommand{\pts}{\operatorname{pts}}
\newcommand{\pic}{\operatorname{Pic}}
\renewcommand{\O}{\mathcal{O}}
\newcommand{\n}{{\boldsymbol{n}}}

\newcommand{\betab}{{\boldsymbol{\beta}}}

\renewcommand{\S}[1]{S^{[#1]}}
\newcommand{\z}[1]{\Z^{[#1]}}
\renewcommand{\i}[1]{\I^{[#1]}}

\newcommand{\id}{\operatorname{id}}
\newcommand{\pr}{\operatorname{pr}}
\newcommand{\LL}{\mathbb{L}^{\bullet}}

\newcommand{\Fb}{F^{\bullet}}
\newcommand{\Ab}{A^{\bullet}}

\newcommand{\Eb}{E^{\bullet}}

\newcommand{\Fbv}{F^{\bullet \vee}}

\newcommand{\dR}{\mathbf{R}}
\newcommand{\dL}{\mathbf{L}}
\renewcommand{\hom}{\mathcal{H}om}
\newcommand{\ext}{\mathcal{E}xt}
\newcommand{\Hom}{\operatorname{Hom}}
\newcommand{\Ext}{\operatorname{Ext}}
\newcommand{\cone}{\operatorname{Cone}}
\newcommand{\Coh}{\operatorname{Coh}}
\newcommand{\tr}{\operatorname{tr}}

\newcommand{\q}{\operatorname{q}}
\newcommand{\coker}{\operatorname{coker}}

\newcommand{\rank}{\operatorname{Rank}}
\newcommand{\mM}{\mathcal{M}}
\newcommand{\mMb}{\mathcal{M}^{\mathcal{L}}}
\newcommand{\mMw}{\mathcal{M}^{\omega_S}}
\newcommand{\cE}{\mathcal{E}}
\newcommand{\eE}{\mathbb{E}}
\newcommand{\eEb}{\overline{\mathbb{E}}}
\newcommand{\ch}{\operatorname{ch}}
\newcommand{\DT}{\operatorname{DT}}
\newcommand{\VW}{\operatorname{VW}}
\newcommand{\vir}{\operatorname{vir}}

\renewcommand{\v}[1]{V^{[#1]}}

\renewcommand{\c}{\mathsf{c}}
\newcommand{\sP}{\mathsf{P}}
\newcommand{\sQ}{\mathsf{Q}}
\newcommand{\sA}{\mathsf{A}}
\newcommand{\F}{\mathcal{F}}

\renewcommand{\t}{\mathbf{t}}
\newcommand{\s}{\mathbf{s}}
\newcommand{\J}{\mathbb{J}}
\newcommand{\G}{\mathcal{G}}
\newcommand{\fix}{\operatorname{fix}}
\newcommand{\mov}{\operatorname{mov}}
\newcommand{\Ebf}{E^{\bullet, \fix}}
\newcommand{\Ebm}{E^{\bullet,\mov}}
\renewcommand{\d}{\operatorname{d}}
\newcommand{\T}{\mathcal{T}}
\newcommand{\N}{\mathcal{N}}
\newcommand{\D}{\mathcal{D}}

\newcommand{\red}{\operatorname{red}}

\newcommand{\ob}{\operatorname{ob}}

\newcommand{\cP}{\mathcal{P}}
\newcommand{\sN}{\mathsf{N}}
\newcommand{\sG}{\mathsf{G}}

\newcommand{\sK}{\mathsf{K}}
\newcommand{\sT}{\mathsf{T}}
\newcommand{\sE}{\mathsf{E}}

\newcommand{\cL}{\mathcal{L}}
\newcommand{\kk}[1]{\sK^{[#1]}}
\newcommand{\mm}{\mathfrak{m}}
\newcommand{\nn}{\mathfrak{n}}
\newcommand{\jj}{\mathfrak{J}}
\newcommand{\bp}{\overline{p}}

\title[Localized DT theory]{Localized Donaldson-Thomas theory of surfaces}

\author{Amin Gholampour and Artan Sheshmani and Shing-Tung Yau}
\begin{document}

\maketitle

\begin{abstract}
Let $S$ be a projective simply connected complex surface and $\cL$ be a line bundle on $S$. We study the moduli space of stable compactly supported 2-dimensional sheaves on the total spaces of $\cL$. The moduli space admits a $\C^*$-action induced by scaling the fibers of $\cL$. We identify certain components of the fixed locus of the moduli space with the moduli space of torsion free sheaves and the nested Hilbert schemes on $S$. We define the localized Donaldson-Thomas invariants of $\cL$ by virtual localization in the case that $\cL$ twisted by the anti-canonical bundle of $S$ admits a nonzero global section. When $p_g(S)>0$, in combination with Mochizuki's formulas, we are able to express the localized DT invariants in terms of the invariants of the nested Hilbert schemes defined by the authors in \cite{GSY17a}, the Seiberg-Witten invariants of $S$, and the integrals over the products of Hilbert schemes of points on $S$.
When $\cL$ is the canonical bundle of $S$, the Vafa-Witten invariants defined recently by Tanaka-Thomas, can be extracted from these localized DT invariants. VW invariants are expected to have modular properties as predicted by S-duality.

\end{abstract}

\setcounter{tocdepth}{3}
\tableofcontents

\section{Introduction} 
\subsection{Overview}
The Donaldson-Thomas invariants of 2-dimensional sheaves in projective nonsingular (Calabi-Yau) threefolds are expected to have modular properties through S-duality considerations (\cite{DM11,VW94,  GS13, GST14}). These invariants are very difficult to compute in general due to lack of control over the singularity of  surfaces supporting these sheaves. To make the situation more manageable, we consider the total space of a line bundle $\cL$ over a fixed nonsingular projective surface $S$. We then study the moduli space of $h$-stable 2-dimensional compactly supported sheaves $\cE$ such that $c_1(\cE)=r[S]$, where $[S]$ is the class of the 0-section and $h=c_1(\O_S(1))$.

To define DT invariants of $\cL$ we have to overcome two main obstacles: \begin{enumerate}[1.]
\item Construct a perfect obstruction theory over the moduli space, which contains no trivial factor in its obstruction sheaf\footnote{Otherwise, the DT invariants would vanish.},
\item If $H^0(\cL)\neq 0$ then the moduli space is not compact and hence one cannot expect to get a well-defined virtual fundamental class from 1. 
\end{enumerate}

For 1, we do not allow strictly semistable sheaves in the moduli space, and we assume that the line bundle $\cL\otimes \omega_S^{-1}$ admits a nonzero global section, where $\omega_S$ is the canonical bundle of $S$. The latter condition guarantees that higher obstruction spaces of stable sheaves under consideration either vanish (if $\cL\neq \omega_S$) or can be ignored (if $\cL=\omega_S$), and in any case \cite{T98, HT10} provide the moduli space with a natural perfect obstruction theory. We assume that $H^1(\O_S)=0$, and then construct a reduced perfect obstruction theory out of the natural one by removing from its obstruction sheaf a trivial factor of rank $p_g(S)$.

For 2, we consider the $\C^*$-action on the moduli space induced by scaling the fibers of $\cL$. The fixed set of the moduli space is compact and the fixed part of the reduced perfect obstruction theory above leads to a reduced virtual fundamental class over this fixed set \cite{GP99}. We define two types of Donaldson-Thomas invariants by integrating against this class.
The study of these invariants completely boils down to understanding the fixed set of the moduli space and also the fixed and moving parts of the reduced perfect obstruction theory.
By restricting to the fixed set of the moduli space, we have much more control over the possible singularities of the supporting surfaces: the only singularities that can occur are the  thickenings of the zero section along the fibers of $\cL$.

\subsection{Main results} \label{mainresults}
We fix some symbols and notation before expressing the results. Let $S$ be a nonsingular projective  surface with $H^1(\O_S)=0$ and let $h=c_1(\O_S(1))$. Let  $\q:\cL\to S$ be a line bundle on $S$ so that $H^0(\cL\otimes \omega_S^{-1})\neq 0$. Let $$v=(r,\gamma,m) \in \oplus_{i=0}^2 H^{2i}(S,\mathbb{Q})$$ be a Chern character vector with $r\ge 1$,  and $\mMb_h(v)$ be the moduli space of compactly supported 2-dimensional stable sheaves $\cE$ on $\cL$ such that $\ch(\q_*\cE)=v$. Here stability is defined by means of the slope of $\q_*\cE$ with respect to the polarization $h$, and we assume $\gcd(r,\gamma\cdot h)=1$.  %We provide $\mMb_h(v)$ with a perfect obstruction theory by reducing the natural perfect obstruction theory given by \cite{T98}. We define the DT invariants $\DT_h(v)$ by taking the virtual Euler number of the $\C^*$-fixed locus of $\mMb_h(v)$ (see Definiton \ref{DThv4} and Remarks \ref{virnum}, \ref{vafawitten}).

The $\C^*$-fixed locus $\mMb_h(v)^{\C^*}$ consists of sheaves supported on $S$ (the zero section of $\cL$) and its thickenings. As discussed above, we show that  $\mMb_h(v)^{\C^*}$ carries a reduced virtual fundamental class denoted by $[\mMb_h(v)^{\C^*}]^{\vir}_{\red}$ (Theorem \ref{prop:red}).
In this paper we study two types of DT invariants

\begin{align*}
\DT^{\cL}_h(v;\alpha)&=\int_{[\mMb_h(v)^{\C^*}]^{\vir}_{\red}} \frac{\alpha}{e(\text{Nor}^{\vir})}\in \mathbb{Q}[\s,\s^{-1}]\quad \quad \alpha \in H^*_{\C^*}(\mMb_h(v)^{\C^*},\mathbb{Q})_\s,\\
\DT^{\cL}_h(v)&=\chi^{\vir}(\mMb_h(v)^{\C^*})\in \ZZ,
\end{align*} where $\text{Nor}^{\vir}$ is the virtual normal bundle of $\mMb_h(v)^{\C^*}\subset \mMb_h(v)$, $\chi^{\vir}(-)$ is the virtual Euler characteristic \cite{FG10}, and $\s$ is the equivariant parameter. 

If $\cL=\omega_S$ and $\alpha=1$ then $$\DT^{\omega_S}_h(v;1)=\s^{-p_g}\VW_h(v),$$ where $\VW_h(-)$ is the Vafa-Witten invariant defined by Tanaka and Thomas \cite{TT} and are expected to have modular properties (see Section \ref{vafawitten}).

We write $\mMb_h(v)^{\C^*}$ as a disjoint union of several types of components, where each type is indexed by a partition of $r$. 
Out of these component types, there are two types that are in particular important for us. One of them (we call it type I) is identified with $\mM_h(v)$, the moduli space of rank $r$ torsion free stable sheaves on $S$. The other type (we call it type II) can be identified with the nested Hilbert scheme $\S{\n}_\betab$ for a suitable choice of nonnegative integers $\n:=n_1,\dots, n_r$ and effective curve classes $\betab:=\beta_1,\dots, \beta_{r-1}$ in $S$.  Here $\S{\n}_\betab$ is the nested Hilbert scheme on $S$ parameterizing tuples $$(Z_1,Z_2,\dots,Z_r),\quad (C_1,\dots,C_{r-1})$$ where $Z_i\subset S$ is a 0-dimensional subscheme of length $n_i$, and $C_i\subset S$ is an effective divisor with $[C_i]=\beta_i$, and for any $i<r$ \begin{equation} \label{inclusion} I_{Z_{i}}(-C_i)\subset I_{Z_{i+1}}.\end{equation}
If $\beta_1=\cdots=\beta_{r-1}=0$, then $\S{\n}:=\S{\n}_{\betab=0}$ is the nested Hilbert scheme of points on $S$.
The authors have constructed a perfect obstruction theory over $\S{\n}_{\betab}$ in \cite{GSY17a} by studying the deformation/obstruction of the natural inclusions \eqref{inclusion}.
As a result $\S{\n}_{\betab}$ is equipped with a virtual fundamental class denoted by $[\S{\n}_\betab]^{\vir}$. This allows us to define new invariants  for $S$ recovering in particular Poincar\'e invariants of \cite{DKO07}, and (after reduction) stable pair invariants of \cite{KT14}.

The following Theorems are proven in Propositions \ref{(r)}, \ref{threefoldtwofold} and \ref{fixedpart}:
\begin{theorem}
The restriction of $[\mMb_h(v)^{\C^*}]^{\vir}_{\red}$ to the type I component $\mM_h(v)$ is identified with $[\mM_h(v)]^{\vir}_0$ induced by the natural trace free perfect obstruction theory over $\mM_h(v)$.
\end{theorem}

\begin{theorem}

The restriction of $[\mMb_h(v)^{\C^*}]^{\vir}_{\red}$ to a type II component $\S{\n}_\betab$ is identified with  $[\S{\n}_\betab]^{\vir}$ constructed in \cite{GSY17a}.
\end{theorem}
%The reason that types I and II are more interesting for us is that (as proven in Propositions \ref{(r)} and \ref{fixedpart}), the restriction of the fixed part of the perfect obstruction theory of $\mMb_h(v)$ to these component types respectively identified with the natural trace free perfect obstruction theory on $\mM_h(v)$, and the perfect obstruction theory we constructed for $\S{\n}_\betab$ in Theorem \ref{thm1}. Let $\chi^{\vir}(\mM_h(v))$ and $\chi^{\vir}(\S{\n}_{\betab})$ be the virtual Euler numbers with respect to these perfect obstruction theories \cite{FG10}. 

When $r=2$ then types I and II components are the only possibilities. This leads us to the following result (Propositions \ref{virdec}, \ref{(r)}, \ref{fixedpart}):

\begin{theorem} \label{thm4.5}
Suppose that $v=(2,\gamma,m)$. Then, 
\begin{align*} 
\DT^{\cL}_h(v;\alpha)&=\DT^{\cL}_h(v;\alpha)_{{\rm{I}}}+\sum_{n_1,n_2,\beta}\DT^{\cL}_h(v;\alpha)_{{\rm{II}},\S{n_1,n_2}_\beta},\\
\DT^{\cL}_h(v)&=\chi^{\vir}(\mM_h(v))+\sum_{n_1,n_2,\beta}\chi^{\vir}(\S{n_1,n_2}_{\beta}]),\end{align*} 
where the sum is over all $n_1, n_2, \beta$ (depending on $v$ as in Definition \ref{compatible}) for which $\S{n_1,n_2}_\beta$ is a type II component of  $\mMb_h(v)^{\C^*}$, and the indices I and II indicate the contributions of type I and II components to the invariant $\DT^{\cL}_h(v;\alpha)$.
\end{theorem}
% The stability of sheaves imposes a strong condition on $n_1, n_2, \beta$  appearing in the summation in Theorem \ref{thm4.5} (see Definition \ref{compatible}). For example, if $S$ is a generic complete intersection in a projective space, then for any $n_1, n_2, \beta$ for which $\S{n_1,n_2}_\beta$ is a type II component of  $\mMb_h(v)^{\C^*}$, the condition in Theorem \ref{thm1.5} (leading to the vanishing $[\S{n_1,n_2}_\beta]^{\vir}=0$)  is not satisfied (see Proposition \ref{CI}). 
 
 The invariants $\chi^{\vir}(\S{n_1,n_2}_{\beta}])$ and $\DT^{\cL}_h(v;\alpha)_{{\rm{II}},\S{n_1,n_2}_\beta}$ (for a suitable choice of class $\alpha$ e.g. $\alpha=1$) appearing in Theorem \ref{thm4.5} are special types of the invariants $$\sN_S(n_1,n_2,\beta;-)$$ that we have defined in \cite{GSY17a} by integrating against $[\S{n_1,n_2}_\beta]^{\vir}$ (Definition \ref{invs} and Corollary \ref{movfix}). One advantage of this viewpoint is that it enables us to apply some of the techniques that we developed in [ibid]  to evaluate these invariants in certain cases.
 
Mochizuki in \cite{M02} expresses certain integrals against the virtual cycle of $\mM_h(v)$ in terms of Seiberg-Witten invariants and integrals $\sA(\gamma_1, \gamma_2, v;-)$ over the product of Hilbert scheme of points on $S$ (see Section \ref{Moch}). Using this result we are able to find the following expression for our DT invariants (Corollaries \ref{movfix}, \ref{cor:CI}, \ref{nocurves} and Proposition \ref{DT-nested}):

\begin{theorem} \label{thm5} 
Suppose that  $p_g(S)>0$, and $v=(2,\gamma,m)$ is such that $\gamma\cdot h >2K_S \cdot h$, $\gamma\cdot h$ is an odd number, and $\chi(v) :=\int_S v \cdot td_S \ge 1$. Then,
\begin{align*}\label{Moc:2}
\DT^{\cL}_h(v;1)
=&
-\sum_{\begin{subarray}{c}
\gamma_1 + \gamma_2 =\gamma \\
\gamma_1\cdot h < \gamma_2 \cdot h
\end{subarray}}
\mathrm{SW}(\gamma_1) \cdot 2^{2-\chi(v)} \cdot \sA(\gamma_1, \gamma_2, v;\sP_1)+\sum_{n_1,n_2,\beta} \sN_S(n_1,n_2,\beta;\cP_1).\\
\DT^{\cL}_h(v)
=&
-\sum_{\begin{subarray}{c}
\gamma_1 + \gamma_2 =\gamma \\
\gamma_1\cdot h < \gamma_2 \cdot h
\end{subarray}}
\mathrm{SW}(\gamma_1) \cdot 2^{2-\chi(v)} \cdot \sA(\gamma_1, \gamma_2, v;\sP_2)+\sum_{n_1,n_2,\beta} \sN_S(n_1,n_2,\beta;\cP_2).
\end{align*}
Here $\mathrm{SW}(-)$ is the Seiberg-Witten invariant of $S$, $\sP_i$ and $\cP_i$ are certain universally defined (independent of $S$) explicit integrands (see Proposition \ref{DT-nested}), and the second sums in the formulas are over all $n_1, n_2, \beta$ (depending on $v$ as in Definition \ref{compatible}) for which $\S{n_1,n_2}_\beta$ is a type II component of  $\mMb_h(v)^{\C^*}$.

Moreover, if $\cL=\omega_S$ and $S$ is isomorphic to a  $K3$ surface or one of the five types of very general  complete intersections $$(5)\subset \mathbb{P}^3, \; (3,3)\subset \mathbb{P}^4,\; (4,2)\subset \mathbb{P}^4, \; (3,2,2)\subset \mathbb{P}^5,\; (2,2,2,2)\subset \mathbb{P}^6,$$ the DT invariants $\DT^{\omega_S}_h(v;1)$ and $\DT^{\omega_S}_h(v)$ can be completely expressed as the sum of integrals over the product of the Hilbert schemes of points on $S$.

\end{theorem}
In Theorem \ref{thm5}, we can always replace a given vector $v$ by another vector (without changing the DT invariants in the right hand side of formulas), for which the condition in theorem is satisfied (see Remark \ref{replacev}). 

%In the following two cases we will show that the right hand side of Theorem \ref{thm5} can be fully expressed in terms of integrations over the product of the Hilbert schemes of points on $S$ and the SW invariants:
%\begin{itemize}
%\item $K_S\cdot h=0$ in which case we show that the second sum in Theorem \ref{thm5} does not appear, i.e. there is no contribution from the nested Hilbert schemes to the invariant $\DT_h(v)$ (Proposition \ref{DT-nested} part (1)).

%\item $S$ is isomorphic to one of the five types of generic complete intersections $$(5)\subset \mathbb{P}^3, \; (3,3)\subset \mathbb{P}^4,\; (4,2)\subset \mathbb{P}^4, \; (3,2,2)\subset \mathbb{P}^5,\; (2,2,2,2)\subset \mathbb{P}^6,$$ in which case we show that in the second sum in Theorem \ref{thm5}, $\beta=0$ always, and then we apply Theorem \ref{thm3} to express the invariants $\sN_S(n_1,n_2,0;\cP)$ as integrals over the products of Hilbert schemes of points on $S$ (Proposition \ref{DT-nested} part (2)).
%\end{itemize}

\section*{Aknowledgement}
We would like to sincerely thank Yokinubu Toda for sharing his ideas with us regarding the relation of  DT theory of local surfaces with the nested Hilbert schemes and also to Mochizuki's work. We are grateful to Richard Thomas for explaining his recent work with Yuuji Tanaka \cite{TT} and providing us with valuable comments. We would like to thank Martijn Kool for pointing out a mistake in the summation in the definition of $\sA$ on page 25 in the first draft of this paper. We would also like to thank Davesh Maulik, Hiraku Nakajima, Takur\={o} Mochizuki, Alexey Bondal and Mikhail Kapranov for useful discussions. 

A. G. was partially supported by NSF grant DMS-1406788.  A. S. was partially supported by NSF DMS-1607871, NSF DMS-1306313 and Laboratory of Mirror Symmetry NRU HSE, RF Government grant, ag. No 14.641.31.0001. The second author would like to further sincerely thank the Center for Mathematical Sciences and Applications at Harvard University, the center for Quantum Geometry of Moduli Spaces at Aarhus University, and the Laboratory of Mirror Symmetry in Higher School of Economics, Russian federation, for the great help and support. S.-T. Y. was partially supported by NSF DMS-0804454,
NSF PHY-1306313, and Simons 38558.

\begin{conv} If $f:X\to Y$ is a morphism of schemes over $\C$ and $Z$ is any other $\C$-scheme, we usually use the same symbol $f$ to denote the morphism $$f\times \id: X\times Z\to Y\times Z.$$ Moreover, if $\F$ is a coherent sheaf on $Y$, when it is clear from the context, we simply write $\F$ to denote its pullback $f^* \F$ to $X$.  
\end{conv}

\section{Local reduced Donaldson-Thomas Invariants} \label{sec:threefold}
Let $(S, h)$ be a pair of a nonsingular projective surface $S$ with $H^1(\O_S)=0$, and $h:=c_1(\O_S(1))$, and let $$v:=(r,\gamma,m)\in H^{\text{ev}}(S, \mathbb{Q})=  H^0(S) \oplus H^2(S) \oplus H^4(S),$$ with $r\ge 1$. We denote by $\mM_h(v)$ the moduli space of $h$-semistable sheaves on $S$ with Chern character $v$. $\mM_h(v)$ is a  projective scheme. We always assume $v$ is such that slope semistability implies slope stability with respect $h$ for any sheaf on $S$ with Chern character $v$. We also assume $\mM_h(v)$  admits a universal family\footnote{The existence of the universal family is not essential in this paper, but we assume it for simplicity.}, denoted by $\eE$. For example, if $\gcd(r,\gamma\cdot h)=1$, these requirements are the case (see \cite[Corollary 4.6.7]{HL10}). If $p$ is the projection to the second factor of  $S\times \mM_h(v)$, by \cite{T98, HT10} $$\dR\hom_{p}(\eE,\eE)_0[1]$$ is the virtual tangent bundle of a (trace-free) perfect obstruction theory on $\mM_h(v)$, that gives a virtual fundamental class, denoted by 
$[\mM_h(v)]^{\vir}_0$.

Let $\cL$ be a line bundle on $S$ such that \begin{equation}\label{antican} H^0(\cL\otimes \omega_S^{-1})\neq 0, \end{equation} and let
\begin{align*}
 X:= \cL \xrightarrow{\q} S
\end{align*}
be the total space of the canonical line bundle on $S$. 
Note that $X$ is  non-compact with canonical bundle $\omega_X\cong \q^*(\cL^{-1}\otimes \omega_S)$. In particular, $X$ is a Calabi-Yau 3-fold if $\cL=\omega_S$.  Let $z:S\to X$ be the zero section inclusion. 

%\begin{notn} For simplicity, we use the symbols $z$ and $\q$ to indicate respectively the inclusion $z\times \id:S\times B\to X\times B$ and the projection $\q\times \id:X\times B\to S\times B$ for any scheme $B$. 
%\end{notn}

The one dimensional complex torus $\mathbb{C}^{\ast}$
acts on $X$ by the multiplication on the fibers of $\q$, so that  \begin{equation}\label{piox} \q_* \O_X=\bigoplus_{i=0}^{\infty} \cL^{-i}\otimes \t^{-i},\end{equation} where $\t$ denotes the trivial line bundle on $S$ with the $\C^*$-action of weight 1 on the fibers.
Let
\begin{align*}
\Coh_c(X) \subset \Coh(X)
\end{align*}
be the abelian category of 
coherent sheaves on $X$
whose supports are compact. 
The slope function $\mu_h$ on $\Coh_c(X) \setminus \{0\}$
\begin{align*}
\mu_h(\cE) =\frac{c_1(\q_{\ast}\cE) \cdot h}{\rank(\q_{\ast}\cE)} \in 
\mathbb{Q} \cup \{ \infty\}
\end{align*}
determines a slope stability condition on $\Coh_c(X)$\footnote{If $\rank(\q_{\ast}\cE)=0$, then $\mu_h(\cE)=\infty$.}. 
Let $\mMb_h(v)$ be the moduli space 
of $\mu_h$-stable sheaves $\cE \in \Coh_c(X)$ with $\ch(\q_{\ast}\cE)=v.$ For simplicity, we also assume $\mMb_h(v)$ admits a universal family, denoted by $\eEb$. This is again the case if for example $\gcd(r,\gamma\cdot h)=1$ (see \cite[Corollary 4.6.7]{HL10}).%Note that the moduli space $\mM_h(v)$ 
%is an open and closed subscheme of 
%$\mMb_h(v)^{\C^*}$, we will get back to this point shortly.
%Let $\eEb \in \Coh(X \times \mMb_h(v))$
%be the universal family, which exists  by the condition on $v$ and \cite[Corollary 4.6.6]{HL10},  

We  denote by $\bp$ the projection 
from $X \times \mMb_h(v)$ 
to $\mMb_h(v)$. By the condition \eqref{antican} and \cite{T98, HT10}, one obtains a natural perfect obstruction theory $$E^{\bullet}\to \LL_{\mMb_h(v)} $$ on $\mMb_h(v)$ whose virtual tangent bundle is given by the complex\footnote{The truncation functor $\tau^{[i,j]}:D^b(-)\to D^b(-)$ sends a complex $A^\bullet$ in the given derived category of coherent sheaves to the complex $$\cdots 0\to \coker(d^{i-1})\to A^{i+1}\to \cdots \to A^{j-1}\to \ker(d^j)\to 0\cdots. $$ Similarly, the functors $\tau^{\le j}$ (resp. $\tau^{\ge i}$) truncates $A^{\bullet}$ as above from right only (resp. left only). }
$$(\Eb)^{\vee}=\tau^{[1,2]}(\dR \hom_{ \bp}(\overline{\eE}, 
\overline{\eE}))[1].
$$ Note that Serre duality and Hirzerbruch-Riemann-Roch hold for the compactly supported coherent sheaves, even though $X$ is not compact. Since $X$ is a nonsingular threefold, the complex $\dR \hom_{ \bp}(\overline{\eE}, 
\overline{\eE})$ is of perfect amplitude contained in $[0,3]$ \footnote{This means that  $\dR \hom_{ \bp}(\overline{\eE}, 
\overline{\eE})$ is quasi-isomorphic to a complex of vector bundles $A^0\to A^1\to A^2\to A^3$ where $A^i$ is in degree $i$.}. For any closed point $\cE \in \mMb_h(v)$ we know $\Hom(\cE,\cE)=\C$ by the stability of $\cE$. Also, $\Ext^3(\cE,\cE)=\C$ if $\cL=\omega_S$, and $0$ otherwise (by stability and Serre duality). So by basechange and Nakayama lemma (as is \cite[Sections 4.3, 4.4]{HT10}),
$\tau^{[1,2]}( \dR \hom_{ \bp}(\overline{\eE}, \overline{\eE}))$ is of perfect amplitude contained in $[1,2]$. Therefore, $\Eb$ is of perfect amplitude contained in $[-1,0]$, as desired.

 Using Hirzerbruch-Riemann-Roch, we can calculate the rank of $\Eb$: let $\cE$ be a coherent sheaf corresponding to a closed point of $\mMb_h(v)$. Then
\begin{align*}
\rank(\Eb)&=\text{ext}^1(\cE,\cE)-\text{ext}^2(\cE,\cE)\\&=1-\kappa-\sum_{i=0}^3(-1)^i\text{ext}^i(\cE,\cE),
\end{align*} where  $\kappa=1$ if $\cL=\omega_S$, otherwise $\kappa=0$. %The second equality above is because $\text{ext}^0(\cE,\cE)=1$ by the stability of $\cE$, and $\text{ext}^3(\cE,\cE)=1$ if  $\cL=\omega_S$ and otherwise $\text{ext}^3(\cE,\cE)=0$, by Serre duality and the stability of $\cE$.
 Therefore we get
\begin{equation} \label{rankE}
\rank(\Eb)=\begin{cases} 0 & \cL=\omega_S,\\ r^2 c_1(\cL)\cdot(c_1(\cL)-\omega_S)/2+1 & \cL\neq \omega_S.\end{cases}
\end{equation}  Here, we used $\text{td}_1(X)=\q^*(c_1(\cL)-\omega_S)/2$ and $[S]^2=z_*c_1(\cL)$. This perfect obstruction theory is known to be  symmetric  if $\cL=\omega_S$ \cite{B09}. 

By \cite{GP99},
we obtain the  $\mathbb{C}^{\ast}$-fixed 
perfect obstruction theory
\begin{align*}
\Ebf=\left(\left (\tau^{[1,2]}( \dR \hom_{ \bp}(\overline{\eE}, 
\overline{\eE}))[1]
\right)^\vee\right)^{\C^{\ast}}
\to \LL_{\mMb_h(v)^{\mathbb{C}^{\ast}}}. 
\end{align*} over the fixed locus $ \mMb_h(v)^{\mathbb{C}^{\ast}}$. Since the $\C^*$-fixed set of $X$ (i.e. $S$) is projective, we conclude that $\mMb_h(v)^{\mathbb{C}^{\ast}}$ is projective, therefore $E^{\bullet,\fix}$ gives the virtual fundamental class $[\mMb_h(v)^{\mathbb{C}^{\ast}}]^{\vir}$. Define

\begin{equation}\label{DThv3} \widehat{\DT^{\cL}_h}(v;\alpha)=\int_{[\mMb_h(v)^{\mathbb{C}^{\ast}}]^{\rm{vir}}}
\frac{\alpha}{e((E^{\bullet,\mov})^\vee)} \quad \quad \alpha \in H^{*}_{\C^*}(\mMb_h(v),\mathbb{Q})_{\s},\end{equation} where $E^{\bullet,\mov}$ is the $\C^*$-moving part of $\Eb$, and $e(-)$ indicates the equivariant Euler class.

\begin{rmk} \label{notdef}
 Note that $(E^{\bullet,\mov})^\vee$ is the virtual normal bundle of $\mMb_h(v)^{\mathbb{C}^{\ast}}$. If $\mMb_h(v)$ is compact then $\displaystyle \widehat{\DT^{\cL}_h}(v;\alpha)$ will be equal to $\int_{[\mMb_h(v)]^{vir}}\alpha$ via the virtual localization formula \cite{GP99}. This is the case when $c_1(\cL)\cdot h< 0$, as then one can see that all the stable sheaves must be supported (even scheme theoretically!) on the zero section of $\q:X\to S$. Note that if $c_1(\cL)\cdot h\ge 0$, then $\int_{[\mMb_h(v)]^{vir}}\alpha$ is not defined in general.
\end{rmk}

\begin{rmk} If $\cL=\omega_S$ (i.e. $X$ is Calabi-Yau), one can also define the invariants by taking weighted Euler characteristics of the moduli spaces
$\int_{\mMw_h(v)} \nu_{\mMw} \ d\chi$, 
where $\nu_{\mMw}$ is Behrend's constructible function ~\cite{B09} on 
$\mMw_h(v)$. 
By localization this  
coincides with the integration of $\nu_{\mMw}$ 
over the $\mathbb{C}^{\ast}$-fixed locus 
$\mMw_h(v)^{\mathbb{C}^{\ast}}$. These invariants were computed by \cite{TT} and were shown to have modular properties in some interesting examples. If $\mMw_h(v)$ is compact e.g. when $K_S\cdot h <0$ (see Remark \ref{notdef}) then these invariants coincide with the invariants $\widehat{\DT^{\omega_S}_h}(v;1)$ \cite{B09}.
 \end{rmk}

%\begin{rmk}\label{redkon}
%Suppose that $\cE$ is a closed point of $\mMb_h(v)$. Consider the composition of the natural maps $$\Ext^2_X(\cE,\cE)\to \Ext^2_S(\q_*\cE,\q_*\cE)\xrightarrow{\tr}H^2(\O_S).$$  It can be seen that this map is surjective, and hence the obstruction sheaf $h^2((\Eb)^\vee)$ contains a $\C^*$-fixed trivial factor when $p_g(S)>0$, and as a result the invariants $\widehat{\DT}_h(v)=0$ in this case.
%\end{rmk}

%Because of Remark \ref{redkon}
In the case that $p_g(S)>0$ the fixed part of the obstruction theory $\Eb$ contains a trivial factor which causes the invariants $\widehat{\DT^{\cL}_h}(v)$ to vanish; we reduce the obstruction theory $\Eb$ as follows. Define $C^\bullet$ to be the cone of the composition %$$\dR \hom_{\bp }(\eEb, \eEb)_0$$ 
\begin{align*}
\q_{\ast}\dR \hom_{X\times \mMb_h(v)}
(\eEb, \eEb) \xrightarrow{\q_*}
\dR \hom_{S\times \mMb_h(v)}(\q_{\ast} \overline{\eE},
\q_{\ast}\overline{\eE}) \xrightarrow{\tr} 
\O_{S\times \mMb_h(v) },
\end{align*} followed by the derived push forward via the projection $p : S\times \mMb_h(v)\to \mMb_h(v)$. Note that $\q$ is an affine morphism and hence $R^i\q_*=0$ for $i>0$. Then, define
\begin{equation} \label{3fold}
\Eb_{\red}:=\left(\tau^{\le 1}(C^\bullet )
\right)^\vee.
\end{equation}

\begin{lem} \label{perfect}
$\tau^{\le 1}(C^\bullet )$ is of perfect amplitude contained in $[0,1]$. Moreover,
$$h^0(\tau^{\le 1}(C^\bullet))\cong \ext^1_{\bp}(\eEb,\eEb),$$ and $h^1(\tau^{\le 1}(C^\bullet))$ fits into the short exact sequence $$0\to h^1(\tau^{\le 1}(C^\bullet))\to \ext^2_{\bp}(\eEb,\eEb) \to \O_{\mMb_h(v)}^{p_g}\to 0.$$
\end{lem}
\begin{proof}
Let $m \in \mMb_h(v)$ be a closed point corresponding to a stable coherent sheaf $\cE$. 
Restricting the resulting exact triangle \begin{equation}\label{mapcone} \dR\hom_{\bp}(\eEb,\eEb)\xrightarrow{\dR p_*\circ ( \tr\circ \q_*)} \dR p_*\O\to C^\bullet\end{equation}  to this closed point (i.e. derived pullback) and taking cohomology we get the exact sequence 
\begin{align*}
0&\to h^{-1}(C^\bullet|_m)\to  \Hom(\cE,\cE)\xrightarrow{\tr\circ \q_*} H^0(\O_S)\to  h^{0}(C^\bullet|_m)\to \Ext^1(\cE,\cE)\xrightarrow{\tr\circ \q_*} H^1(\O_S)\\ 
&\to h^{1}(C^\bullet|_m)\to  \Ext^2(\cE,\cE)\xrightarrow{\tr\circ \q_*} H^2(\O_S)\to  h^{2}(C^\bullet|_m)\to \Ext^3(\cE,\cE)\to 0. 
\end{align*}
Now since the composition
$$\O_{S\times \mMb_h(v) } \xrightarrow{\id} \q_{\ast}\dR \hom_{X\times \mMb_h(v)}
(\eEb, \eEb) \xrightarrow{\q_*}
\dR \hom_{S\times \mMb_h(v)}(\q_{\ast} \overline{\eE},
\q_{\ast}\overline{\eE}) \xrightarrow{\tr} 
\O_{S\times \mMb_h(v) }$$ is $r\cdot \id$, we see that all the arrows labeled by $\tr\circ \q_*$ in the long exact sequence above are surjective. Combining with the stability of $\cE$, and vanishing $ H^1(\O_S)=0$, we conclude
$h^{-1}(C^\bullet|_m)=0$ and $h^{2}(C^\bullet|_m)=\Ext^3(\cE,\cE)$. 

Now if $\cL \neq \omega_S$ then $\Ext^3(\cE,\cE)=0$ and so $C^\bullet$ is already of perfect amplitude contained in $[0,1]$ (so $C^\bullet =\tau^{\le 1}(C^\bullet)$). If $\cL=\omega_S$ then $\Ext^3(\cE,\cE)=\C$ by Serre duality and stability of $\cE$. So again by basechange and Nakayama Lemma $\tau^{\le 1}(C^\bullet)$ is of perfect amplitude contained in $[0,1]$, and the first part of Lemma is proven. 

The claim about sheaf cohomologies follows from the long exact sequence of sheaf cohomology (associated to the exact triangle \eqref{mapcone}), the identity $h^i(C^\bullet)=h^i(\tau^{\le 1}(C^\bullet))$ for $i=0,1$, and the fiberwise analysis above.

\end{proof}

\begin{thm} \label{prop:red}
$\left(\Eb_{\red}\right)^{\vee}=\tau^{\le 1}(C^\bullet )$ is the virtual tangent bundle of a perfect obstruction theory over $\mMb_h(v)$.
\end{thm}
\begin{proof}[Proof of Theorem \ref{prop:red} using Li-Tian \cite{LT98} approach] We closely follow the construction of \cite{T98}. We need to show that $\tau^{\le 1}(C^\bullet)$ is a perfect tangent-obstruction complex over $\mMb_h(v)$ in the following sense (\cite[Definition 3.29]{T98}): 

Suppose $B_0$ is an affine scheme over $\C$, $f: B_0 \to \mMb_h(v)$ is a morphism, and $\I$ is an $\O_{B_0}$-module (this is data (3.24) in \cite{T98}). 
Let $\tau^{\le 1}(C^\bullet)\cong C^0\to C^1$ be a 2-term locally free resolution, which is possible by Lemma \ref{perfect}. We have to show that the sheaf cohomologies of the 2-term complex  $$\dL f^*( \tau^{\le 1}(C^\bullet))\overset{\dL}{ \otimes }\I  \cong  f^*C^0\otimes \I\to f^*C^1\otimes \I$$ are respectively the evaluations at $\I$ of the tangent and obstruction functors of $ \mMb_h(v)$ (\cite[Definitions 3.25, 3.27]{T98}), and they also satisfy the compatibility with  basechange. 
Consider the composition
$$ \q_{\ast}\dR \hom_{X\times B_0}
(f^*\eEb, f^*\eEb \otimes \bp^*\I) \xrightarrow{\q_*}
\dR \hom_{S\times B_0}(\q_{\ast} f^* \overline{\eE},
\q_{\ast}f^* \overline{\eE}\otimes p^*\I) \xrightarrow{\tr} 
p^* \I,$$ where $p: S\times B_0\to B_0$ and $\bp: X\times B_0\to B_0$ and we use our convention to denote $\q\times \id$ by $\q$ all over (and so $\bp=p\circ \q$). If we take the mapping cone, apply $\dR p_*$, and take sheaf cohomology, as in the proof of Lemma \ref{perfect}, we get the isomorphism
\begin{equation}\label{ish^0}h^0\big(\dL f^*( \tau^{\le 1}(C^\bullet))\overset{\dL}{ \otimes }\I\big)\cong \ext^1_{\bp}(f^*\eEb,f^*\eEb\otimes \bp^* \I),\end{equation} and the short exact sequence \begin{equation}\label{sesh^1}0\to h^1\big(\dL f^*( \tau^{\le 1}(C^\bullet))\overset{\dL}{ \otimes }\I\big)\to \ext^2_{\bp}(f^*\eEb,f^*\eEb\otimes \bp^* \I) \xrightarrow{\tr\circ \q_*} \I^{\oplus p_g}\to 0.\end{equation}
Note that here we used the fact $$h^i\big(\dL f^*( \tau^{\le 1}(C^\bullet))\overset{\dL}{ \otimes }\I\big)\cong h^i\big(\dL f^*(C^\bullet)\overset{\dL}{ \otimes }\I\big),\qquad i=0,1,$$ which is true because  $\tau^{\le 1}(C^\bullet)$ is of perfect amplitude contained in $[0,1]$ by Lemma \ref{perfect}.

\cite[Prop 3.26]{T98} proves that  $\ext^1_{\bp}(f^*\eEb,f^*\eEb\otimes \bp^* -)$ is the tangent functor for $f:B_0\to \mMb_h(v)$. Therefore,  \eqref{ish^0} implies that $$h^0\big(\dL f^*( \tau^{\le 1}(C^\bullet))\overset{\dL}{ \otimes }-\big)$$ is the tangent functor for $f:B_0\to \mMb_h(v)$.

Next, using the collapse of Tor-$\ext$ spectral sequence as in the proof of \cite[Theorem 3.28]{T98}, $\ext^2_{\bp}(f^*\eEb,f^*\eEb\otimes \bp^* \I)\cong \ext^2_{\bp}(f^*\eEb,f^*\eEb)\otimes  \I$. So by \eqref{sesh^1} and an analog of the short exact sequence in Lemma \ref{perfect} over $B_0$
\begin{equation} \label{Iout} h^1\big(\dL f^*( \tau^{\le 1}(C^\bullet))\overset{\dL}{ \otimes }\I\big)\cong h^1\big(\dL f^*( \tau^{\le 1}(C^\bullet))\big)\otimes \I.\end{equation}
By \cite[Theorem 3.28]{T98} $\ext^2_{\bp}(f^*\eEb,f^*\eEb)$ is an obstruction sheaf for $f:B_0\to \mMb_h(v)$ (in the sense of the following paragraph). Our goal is to show that $$h^1\big(\dL f^*( \tau^{\le 1}(C^\bullet))\big)$$ is also an obstruction sheaf for $f:B_0\to \mMb_h(v)$. 

Let $B_0\subset B\subset B_1$ be closed immersions of $B_0$-schemes over $\C$. We denote the ideals of $B_0\subset B$, $B_0\subset B_1$, $B\subset B_1$ by $\nn$, $\mm$ and $\jj$, respectively, and suppose that $\mm\cdot \jj=0$.  We use the same symbols to denote the pullbacks of these ideals via $p$ and $\bp$. Let $\G_0$ be a sheaf on $B_0\times X$ flat over $B_0$ corresponding to a morphism $f:B_0\to  \mMb_h(v)$, and $\G$ be a sheaf on $B\times X$ flat over $B$ extending $\G_0$. Note that $\q$ is an affine morphism and hence $R^i\q_*=0$ for $i>0$, so by flat basechange $\q_*\G_0$ and $\q_*\G$ remain flat and $\q_*\G|_{B_0\times S}=\q_*\G_0$. By \cite[Proposition 3.13]{T98}, the obstruction for extending $\G$  (respectively $\q_* \G$) to a sheaf on $B_1\times X$ (respectively $B_1\times S$) flat over $B_1$ lies in $$\ob(\G, B,B_1)\in \Ext^2_{X\times B_0}(\G_0,\G_0\otimes \jj),\ \text{(resp. } \ob(\q_*\G, B,B_1)\in \Ext^2_{S\times B_0}(\q_* \G_0,\q_*\G_0\otimes \jj) \text{)}.$$ We will use the abreviations  $\ob(\G)$ and  $\ob(\q_*\G)$ to denote these classes. By definition, $\ob(\G)=0$ (resp. $\ob(\q_* \G)=0$) if and only if there is an extension of $\G$ (resp. $\q_* \G$) over $X\times B_1$ (resp. $S\times B_1$) which is flat over $B_1$. Theorem \cite[Theorem 3.28]{T98} then shows that (as an application of the collapse of Leray spectral sequence) 
\begin{equation} \label{GamB0}\Gamma_{B_0}\big(\ext^2_{\bp}(f^*\eEb,f^*\eEb)\otimes \jj\big)\cong  \Ext^2_{X\times B_0}(\G_0,\G_0\otimes \jj),\end{equation} from which it follows that 
 $\ob(\G) \in \Gamma_{B_0}\big(\ext^2_{\bp}(f^*\eEb,f^*\eEb)\otimes \jj \big)$. The compatibility with basechange follows from basechange property of $\ext_{\bp}^i$.

We will prove the following lemma after finishing the proof of the proposition:
\begin{lem} \label{tech}
Under the natural map $$\Ext^2_{X\times B_0}(\G_0,\G_0\otimes \jj)\xrightarrow{\q_*}\Ext^2_{S\times B_0}(\q_*\G_0,\q_*\G_0\otimes \jj),$$
we have $\q_*\ob(\G)=\ob(\q_*\G).$
\end{lem}
 By \cite[Theorem 3.23]{T98}, 
the obstruction for deforming $\det (q_* \G)$ is given by $\tr(\ob(\q_*\G))$. However, there is no obstruction for deforming line bundles, and therefore $\tr(\ob(\q_*\G))=0$. By Lemma \ref{tech} this means that $\tr(\q_*\ob( \G))=0,$  or equivalently 
 $$\q_* \ob(\G) \in \Ext^2_{S\times B_0}(\q_*\G_0,\q_*\G_0\otimes \jj)_0,$$ %where $$\q_*: \Ext^2_{X\times B_0}(\q_*f^*\eEb, \q_*f^*\eEb\otimes I)\to  \Ext^2_{S\times B_0}(f^*\eEb,f^*\eEb\otimes I)$$ is the natural map. 
and this means that $\ob(\G) \in \ker(\tr \circ \q_*  )$, which by \eqref{Iout}, \eqref{GamB0}, and the short exact sequence \eqref{sesh^1}, gives $$\ob(\G) \in \Gamma_{B_0}\big(h^1(\dL f^*( \tau^{\le 1}(C^\bullet))\otimes \jj\big).$$ This completes the proof of $h^1\big(\dL f^*( \tau^{\le 1}(C^\bullet))\big)$ is an obstruction sheaf for $f:B_0\to \mMb_h(v)$. 
\end{proof}

\begin{proof}[Proof of Lemma \ref{tech}] Suppose that $\G_1$ is a $B_1$-flat lift of $\G$. As in the proof of \cite[Proposition 3.13]{T98} we have short exact sequences $0\to \jj\to \mm\to \nn\to 0,$ and \begin{equation}\label{exten1} 0\to \G\otimes \nn\to \G\to \G_0\to 0,\quad 0\to \G\otimes \mm\to \G_1\to \G_0\to 0.\end{equation} Since $R^i\q_*=0$ for $i>0$, we get the corresponding short exact sequences \begin{equation}\label{exten2} 0\to \q_*\G\otimes \nn\to \q_*\G\to \q_*\G_0\to 0,\quad 0\to \q_*\G\otimes \mm\to \q_*\G_1\to \q_*\G_0\to 0,\end{equation} 
$$0\to \G_0\otimes \jj\to \G\otimes \mm \to \G\otimes \nn\to 0,\quad 0\to \q_*\G_0\otimes \jj\to \q_*\G\otimes \mm\to \q_*\G\otimes \nn \to 0.$$
Applying $\Hom(\G_0,-)$ and $\Hom(\q_*\G_0,-)$ to the last two sequences above and using the functoriality of $\q_*$ we get the following commutative diagram with exact rows:
$$\xymatrix{
\Ext^1(\G_0,\G\otimes \mm) \ar^{\q_*}[d]\ar[r] &\Ext^1(\G_0,\G\otimes \nn) \ar^{\q_*}[d] \ar^{\partial}[r] & \Ext^2(\G_0,\G_0\otimes \jj) \ar^{\q_*}[d] \\  \Ext^1(\q_*\G_0,\q_*\G\otimes \mm) \ar[r] &\Ext^1(\q_*\G_0,\q_*\G\otimes \nn) \ar^{\partial}[r]& \Ext^2(\q_*\G_0,\q_*\G_0\otimes \jj).} 
$$ In particular we get $ \partial \circ \q_*=\q_*\circ \partial $. Let $e\in \Ext^1(\G_0,\G\otimes \nn)$ be the class of the first extension in \eqref{exten1},  and $e'\in \Ext^1(\q_*\G_0,\q_*\G\otimes \nn)$ be the class of the first extension in \eqref{exten2}. By the naturality of $\q_*$ we have $e'=\q_*(e)$. By  \cite[Proposition 3.13]{T98} $\ob(\G_0)=\partial(e)$ and $$\ob(\q_*\G_0)=\partial(e')=\partial(\q_*(e))=\q_*(\partial(e))=\q_*(\ob(\G_0)).$$
\end{proof}

\begin{proof}[Proof of Theorem \ref{prop:red} using Behrend-Fantechi 
\cite{BF97} approach] By Lemma \ref{perfect} we know that $\Eb_{\red}$ is of perfect amplitude contained in $[-1,0]$. It suffices to show that there exists a map $\theta: \Eb_{\red}\to \LL_{\mMb_h(v)}$ in derived category  that defines an obstruction theory, i.e. $h^0(\theta)$ is an isomorphism and and $h^{-1}(\theta)$ is surjective. As usual it suffices to work with the truncation $\tau^{\ge -1}$ of the cotangent complex and this is what we mean by $ \LL_{-}$ in this proof. Again we use the fact that 
the composition $$\O_{S\times \mMb_h(v) } \xrightarrow{\id} \q_{\ast}\dR \hom_{X\times \mMb_h(v)}
(\eEb, \eEb) \xrightarrow{\q_*}
\dR \hom_{S\times \mMb_h(v)}(\q_{\ast} \overline{\eE},
\q_{\ast}\overline{\eE}) \xrightarrow{\tr} 
\O_{S\times \mMb_h(v) }$$ is $r\cdot \id$. This implies that the composition $\tr\circ \q_*$ splits and as a result after applying $Rp_*$ we get the isomorphism 
\begin{equation} \label{splittt}
\dR \hom_{\bp}
(\eEb, \eEb)\cong C^\bullet[-1] \oplus \dR p_* \O_{S\times \mMb_h(v) } \cong C^\bullet[-1] \oplus \O_{ \mMb_h(v) } \oplus \O_{ \mMb_h(v)}^{p_g}[-2].
\end{equation} 
Applying truncation functors to both sides of this splitting it is easy to see that
\begin{equation} \label{splittt1}
\tau^{[1,2]}(\dR \hom_{\bp}
(\eEb, \eEb))\cong \tau^{\le 1}(C^\bullet)[-1] \oplus \O_{ \mMb_h(v)}^{p_g}[-2].
\end{equation} 
Now there is a map $\alpha: \big(\tau^{[1,2]}(\dR \hom_{\bp}(\eEb, \eEb))\big)^{\vee}[1] \to  \LL_{\mMb_h(v)}$ constructed in  \cite[(4.10)]{HT10} by means of the truncated Atiyah class $$A(\eEb) \in \Ext^1_{\mMb_h(v)\times X}(\eEb,\eEb\overset {\dL}{\otimes}  \LL_{\mMb_h(v)\times X})$$ and an application of the truncation functor $\tau^{[1,2]}$. This
together with the splitting \eqref{splittt1} gives a map $$\theta:\Eb_{\red}= \big(\tau^{\le 1}(C^\bullet)\big)^\vee \to \LL_{\mMb_h(v)}.$$
It remains to show that $\theta$ is an obstruction theory. For this we use the criterion in \cite[Theorem 4.5]{BF97} and the fact that it is already proven that $\alpha$ is an obstruction theory in the last part of \cite[Section 4.4]{HT10}.

The question of being an obstruction theory is local in nature, so let $B_0\subset B$ be a closed immersion of affine schemes over $\C$ with the ideal sheaf $I$ such that $I^2=0$, and let $\G_0$ be a sheaf on $B_0\times X$ flat over $B_0$ corresponding to a morphism $f:B_0\to  \mMb_h(v)$. Let $\bp:X\times B_0 \to B_0$ and $p:S\times B_0\to B_0$ be the obvious projections. We have the chain of morphisms $$f^* \Eb_{\red}\xrightarrow{f^*\theta} f^*\LL_{\mMb_h(v)}\to \LL_{B_0}.$$ The pullback of the Kodaira-Spencer class $\kappa(B_0/B)\in \Ext^1(\LL_{B_0}, I)$ via the second arrow gives the obstruction class $\varpi(f) \in \Ext^1(f^*\LL_{\mMb_h(v)},I)$ for extending the map $f$ to $B$.
Pulling back further via the first arrow we get $\theta^*\varpi(f) \in \Ext^1(f^*\Eb_{\red},I)$. By  \cite[Theorem 4.5]{BF97} we have to show that $\theta^*\varpi(f)=0$ if and only if $f$ can be extended to $B$, and in this case the extensions form a torsor over $\Hom(f^*\Eb_{\red},I)$

Similarly pulling back $\varpi(f)$ via $f^*\alpha$ we get 
\begin{align*}\alpha^*\varpi(f) \in &\Ext^1\big(f^* \big(\tau^{[1,2]}(\dR \hom_{\bp}(\eEb, \eEb))\big)^{\vee}[1],I\big)\cong \\
 &\mathbb{H}^2\big(\tau^{[1,2]}(\dR \hom_{\bp}(f^*\eEb, f^*\eEb\otimes \bp^*I))\big)\cong\\ 
&\Ext^2(\G_0,\G_0\otimes \bp^*I),\end{align*} where $\mathbb{H}^2$ denotes the hypercohomology and the isomorphisms are established in \cite[Section 4.4]{HT10} using the collapse of the Leray spectral sequence (and here is where $B_0$ affine is needed!). Taking hypercohomology from both sides of \eqref{splittt1} and identifications above, we see that $\theta^*\varpi(f)$ is the $(\tr\circ \q_*)$-free part of $\alpha^*\varpi(f)$ (i.e. the part corresponding to the first summand in decomposition \eqref{splittt1}). But by Lemma \ref{tech2} below $\alpha^*\varpi(f)$ is the same as its own $(\tr\circ \q_*)$-free part, therefore $\theta^*\varpi(f)=\alpha^*\varpi(f)$. Since $\alpha$ is an obstruction theory, by  \cite[Theorem 4.5]{BF97}, $\theta^*\varpi(f)=\alpha^*\varpi(f)=0$ if and only if $f$ can be extended to $B$, and in this case the extensions form a torsor over $$\Hom(f^* \big(\tau^{[1,2]}(\dR \hom_{\bp}(\eEb, \eEb))\big)^{\vee}[1],I)\cong\Hom(f^*\Eb_{\red},I),$$ where the isomorphism is again by applying hypercohomology to \eqref{splittt1}.

%Using the decomposition $\LL_{\mMb_h(v)\times X}\cong \bp^* \LL_{\mMb_h(v)}\oplus p_X^*\LL_{X}$, the Atiyah class is mapped to 
%\begin{align*}
%\Ext^1_{\mMb_h(v)\times X}(\eEb,\eEb\overset {\dL}{\otimes}  \bp^* \LL_{\mMb_h(v)})&\cong \Ext^1_{\mMb_h(v)}(\dR\hom_{\bp}(\eEb,\eEb \otimes p_X^* \omega_X[3]), \LL_{\mMb_h(v)})\\&\cong
%\Hom_{\mMb_h(v)}(\dR\hom_{\bp}(\eEb,\eEb)^\vee[1], \LL_{\mMb_h(v)}),
%\end{align*} where the first isomorphism is adjunction and the second is Grothendieck-Verdier duality.
%So $\text{At} (\eEb)$ induces the map $\dR\hom_{\bp}(\eEb,\eEb)^\vee[-1]\to \LL_{\mMb_h(v)}$. By 
%\eqref{splitt}, we get a map  $(C^\bullet)^\vee[1] \to \LL_{\mMb_h(v)}$. which 
\end{proof}

\begin{lem} \label{tech2}
$\q_*(\alpha^*\varpi(f))\in  \Ext^2(\q_*\G_0,\q_*\G_0\otimes p^*
I)_0.$ \end{lem}
\begin{proof} Define $$X_{B_0}:=X\times B_0,\quad  X_{B}:=X\times B,\quad S_{B_0}:=S\times B_0,\quad S_{B}:=S\times B.$$ 
Let $i_{X}: X_{B_0} \hookrightarrow X_{B_0}\times X_{B_0}$ and $i_{S}: S_{B_0} \hookrightarrow S_{B_0}\times S_{B_0}$ be the diagonal embeddings,  and $j_X: X_{B_0}\times X_{B_0} \hookrightarrow X_{B_0}\times X_B$ and $j_S: S_{B_0}\times S_{B_0} \hookrightarrow S_{B_0}\times S_B$ be the natural inclusions. Then define $$H_X:=\dL j_X^*\; j_{X*} \O_{\Delta_{ X_{B_0}}},\qquad H_S:=\dL j_S^*\; j_{S*} \O_{\Delta_{ S_{B_0}}}.$$ 
Using  the cartesian diagram 
$$\xymatrix{X_{B_0}\times X_{B_0} \ar[d]^-{\widetilde{\q}} \ar@{^(->}[r]^-{j_X} & X_{B_0}\times X_{B} \ar[d]^-{\widetilde{\q}}\\ 
S_{B_0}\times S_{B_0}  \ar@{^(->}[r]^-{j_S} & S_{B_0}\times S_{B},}$$ where $\widetilde{\q}:=(\q,\q)$, the fact $\O_{\Delta_{X_0}} =\widetilde{\q}^*\O_{\Delta_{S_0}}$, and flatness of $\q$ and hence $\widetilde{\q}$, we get  
\begin{equation} \label{qtilde} H_X=\widetilde{\q}^* H_S. \end{equation}
Huybrechts and Thomas define the universal obstruction class (\cite[Definition 2.8]{HT10}) $$\varpi_X:=\varpi(X_B/X_{B_0})\in \Ext^2_{X_{B_0}\times X_{B_0}}(\O_{\Delta_{ X_{B_0}}}, i_{X*}(\bp^*I))$$ as given by the extension class of the exact triangle $$ i_{X*}(\bp^*I)[1]\cong h^{-1}(H_X)[1]\to \tau^{\ge -1}(H_X)\to h^0(H_X)\cong \O_{\Delta_{ X_{B_0}}},$$ in which the first isomorphism is established in \cite[Lemma 2.2]{HT10} and the second isomorphism is given by the adjunction. The  universal obstruction class $$\varpi_S:=\varpi(S_B/S_{B_0})\in \Ext^2_{S_{B_0}\times S_{B_0}}(\O_{\Delta_{ S_{B_0}}}, i_{S*}(p^*I))$$ is defined similarly by using $H_S$ instead of $H_X$. By \eqref{qtilde} we have
%By definition, it is evident that the universal obstructions
%$$\varpi(X_B/X_{B_0})\in  \Ext^2(\O_{\Delta_{X_{ B_0}}},i_{\Delta_{X_{ B_0}}*}(\bp^*I)),\quad \varpi(S_B/S_{B_0})\in  
%\Ext^2(\O_{\Delta_{S_{ B_0}}},i_{\Delta_{S_{ B_0}}*}(p^*I))$$ are basechanges of $\varpi(B/B_0)$ and so in particular, 
\begin{equation}\label{varpiq}  \varpi_X=\widetilde{\q}^* \varpi_S.\end{equation}
Thinking of $\varpi_X$ and $\varpi_S$ as Fourier-Mukai kernels, and acting them respectively on  $\G_0$ and $\q_*\G_0$, by  \cite[Thm 2.9, Cor 3.4]{HT10} we obtain the obstruction classes  
$$\varpi_X(\G_0)\in \Ext^2_X(\G_0,\G_0\otimes \bp^*I),\quad \varpi_S(\q_* \G_0) \in \Ext^2_S(\q_*\G_0,\q_*\G_0\otimes p^* I).$$ for deforming these sheaves. By \eqref{varpiq} and the commutative diagram 
$$\xymatrix{X_{B_0} \ar[d]^-{\q}& \ar[l]_-{\pr_1} X_{B_0}\times X_{B_0} \ar[d]^-{\widetilde{\q}} \ar[r]^-{\pr_2} & X_{B_0} \ar[d]^-{\q}\\ 
S_{B_0} & \ar[l]_-{\pr_1} S_{B_0}\times S_{B_0}  \ar[r]^-{\pr_2} & S_{B_0},}$$ where $\pr_1, \pr_2$ are obvious projections to the 1st and 2nd factors, an  application of projection formula gives  
\begin{align}\label{XBSB}\q_* \varpi_X(\G_0)&= \q_*\pr_{2*}(\pr_1^*\G_0 \otimes \varpi_X)\\ \notag
&=\q_*\pr_{2*}(\pr_1^*\G_0 \otimes \widetilde{\q}^*\varpi_S)=\pr_{2*}\widetilde{\q}_*(\pr_1^*\G_0 \otimes \widetilde{q}^*\varpi_S)\\ \notag
&=\pr_{2*}(\widetilde{\q}_*\pr_1^*\G_0 \otimes \varpi_S)  =\pr_{2*}(\pr_1^* \q_*\G_0 \otimes \varpi_S)=\varpi_S(\q_*\G_0). \end{align}

But $\varpi_X(\G_0)=\alpha^*\varpi(f)$ by \cite[Cor 3.4]{HT10} and \cite[Thm 4.5]{BF97} as we already know that $\alpha$ is an obstruction theory. So by \eqref{XBSB} we get 
\begin{equation}\label{varpiqG0} \q_*(\alpha^*\varpi(f))=\varpi_S(\q_*\G_0).  \end{equation}

 By \cite[Theorem 3.23]{T98}, 
the obstruction for deforming the line bundle $\det (q_* \G_0)$ is given by the trace of the obstruction class: \begin{equation} \label{tracezero} \tr(\varpi_S(\q_*\G_0)).\end{equation} However, there are no obstructions for deforming line bundles, and therefore \eqref{tracezero} vanishes, or equivalently $$\varpi_S(\q_*\G_0)\in  \Ext^2_S(\q_*\G_0,\q_*\G_0\otimes p^*
I)_0.$$ Now lemma follows from \eqref{varpiqG0}.

\end{proof}

\begin{rmk} \label{redob} Note that by construction  $\rank(\Eb_{\red})=\rank(\Eb)+p_g(S)$. In particular, when $p_g(S)=0$, we have $\Eb=\Eb_{\red}$. Moreover, the reduction that takes $\Eb$ to $\Eb_{\red}$ only affects the fixed parts of the virtual tangent bundles i.e. $\Ebm=\Ebm_{\red}$. %  (note that by the Serre duality and the non-trivial 
%$\mathbb{C}^{\ast}$-weight on 
%$\omega_S$, the higher obstruction space vanishes 
 %after taking the $\mathbb{C}^{\ast}$-fixed part). Therefore, $\Ebf_0$ gives a perfect obstruction theory over $\mMb_h(v)^{\C^*}$ in this case. In general (i.e. when $P_g(S)>0$), we will not attempt to prove that $\Eb_0$ is a perfect obstruction theory (which would imply $\Ebf_0$ is a perfect obstruction theory by \cite{GP99}); instead, we will show directly that $\Ebf_0$ is a perfect obstruction theory over certain components of $\mMb_h(v)^{\C^*}$. 
%\footnote{ {\color{red}this footnote is not needed!!!}By the Serre duality and the non-trivial 
%$\mathbb{C}^{\ast}$-weight on 
%$\omega_S$, the higher obstruction space vanishes 
% after taking the $\mathbb{C}^{\ast}$-fixed part {\color{red}????? details TBD}.} 
% Let $f:B_0\to \mMb_v(h)$ be a morphism of schemes
%Note that $\q$ is an affine morphism and hence $R^i\q_*=0$ for $i>0$. We simply use the symbol $f$ to denote the morphism $\id \times f$.
\end{rmk}

By  Theorem \ref{prop:red} and \cite{GP99}  we get
\begin{cor}
$\Ebf_{\red}$ gives a perfect obstruction theory over $\mMb_h(v)^{\C^*}$, and hence a virtual fundamental class $$[\mMb_h(v)^{\C^*}]^{
\vir}_{\red}\in A_{*}(\mMb_h(v)^{\C^*}).$$
\end{cor} \qed

%Instead of working with the invariant $\DT_h(v)$ defined in \eqref{DThv}, and the invariant $\widehat{\DT}_h(v)$ defined in \eqref{DThv3} 

In the rest of the paper, we will study the invariants defined below:

\begin{defi} \label{DThv4}
We can define two types of DT invariants
\begin{align*}
\DT^{\cL}_h(v;\alpha)&:=
\int_{[\mMb_h(v)^{\C^*}]_{\red}^{\vir}}
\frac{\alpha}{e((E^{\bullet,\mov})^\vee)}\in \mathbb{Q}[\s,\s^{-1}] \quad \quad \alpha \in H^{*}_{\C^*}(\mMb_h(v),\mathbb{Q})_{\s}, \\
\DT^{\cL}_h(v)&:=
\int_{[\mMb_h(v)^{\C^*}]_{\red}^{\vir}} 
c((E^{\bullet,\fix}_{\red})^\vee) \in \ZZ. 
\end{align*} 
Here $e(-)$ denotes the equivariant Euler class, $\s$ is the equivariant parameter, and $c(-)$ denotes the total Chern class. Note that $\Ebm=\Ebm_{\red}$ by Remark \ref{redob}.
\end{defi}

\begin{rmk}  \label{alpha1}
The invariant $\DT_h^{\cL}(v;\alpha)$ is the reduced version of the invariant $\widehat{\DT^{\cL}_h}(v;\alpha)$ given in \eqref{DThv3}. If $\alpha=1$ then it can be seen easily that 
$$\DT^{\cL}_h(v;1)\cdot \s^{\rank(\Eb_{\red})} \in \mathbb{Q},$$ where $\rank(\Eb_{\red})$ is given by \eqref{rankE} and \eqref{3fold}. In particular, if $\cL=\omega_S$ then $\rank(\Eb_{\red})=p_g(S)$.
\end{rmk}

\begin{rmk}\label{virnum}The definition of the invariant $\DT^{\cL}_h(v)$ is motivated by Fantechi-G\"ottsche's virtual 
Euler characteristic  ~\cite{FG10}. $\DT^{\cL}_h(v)$ is the virtual Euler number of $\mMb_h(v)^{\C^*}$:
$$\DT^{\cL}_h(v)=\chi^{\vir}(\mMb_h(v)^{\C^*}).$$
If $\mMb_h(v)^{\C^*}$ is nonsingular with 
expected dimension, then $\DT^{\cL}_h(v)$ coincides with the
topological Euler characteristic of $\mMb_h(v)^{\C^*}$. 
\end{rmk}

\subsection{Vafa-Witten invariants} \label{vafawitten}
Motivated by Vafa-Witten equation and S-duality conjecture \cite{VW94}, Tanaka and Thomas \cite{TT} define Vafa-Witten invariants by constructing a symmetric perfect obstruction theory over the moduli space of Higgs pairs $(G,\phi)$ on $S$ such that $\tr (\phi )=0$.\footnote{They also fix the determinant of $G$, but by our  $H^1(\O_S)=0$ assumption in this paper, this has no effects here.}
By \cite[Prop 2.2, Lem 2.9]{TT} the moduli space of Higgs pairs is isomorphic to one of our moduli spaces $\mMw_h(v)$.

The moduli space of Higgs pairs is equipped with a $\C^*$-action obtained by scaling $\phi$. This is equivalent to the $\C^*$-action on $\mMw_h(v)$ via the identification above. Over the fixed locus of the moduli space of Higgs pairs the trace of $\phi$ is automatically zero (see \cite[Sections 7.1, 7.3]{TT}), as a result the fixed locus of Tanaka-Thomas' moduli space is identified with  $\mMw_h(v)^{\C^*}$.

The fixed part of Tanaka-Thomas' obstruction theory is equivalent in K-theory to $(\Ebf_{\red})^\vee$ and the moving parts differ (in K-theory)  by the trivial bundle of rank $p_g(S)$ carrying a $\C^*$-action along its fibers. This can be seen by a comparison with \cite[Thm 6.1, Cor 3.18]{TT} (see also Remark \ref{redob} above) as follows. In fact, if the virtual tangent bundle of Tanaka-Thomas theory is denoted by $(E^\bullet_{\perp})^\vee$ then, in K-theory (see also (1.7) in \cite{TT})
  $$\ext^{i+1}_{\bp}(\eEb,\eEb)= h^i((E^\bullet_{\perp})^\vee)+R^i p_* \omega_S + R^{i+1}p_* \O, \qquad i=0,1.$$
Comparing with Lemma \ref{perfect} and using our assumption $H^1(\O_S)=0$, we get in K-theory 
$$h^0((E^\bullet_{\red})^\vee)=h^0((E^\bullet_{\perp})^\vee)+p_* \omega_S,\qquad h^1((E^\bullet_{\red})^\vee)=h^1((E^\bullet_{\perp})^\vee).$$ Since $\omega_S$ carries a non-trivial $\C^*$-weight, the term $p_* \omega_S$ is trivial of rank $p_g(S)$ and contributes only in the moving part of our obstruction theory and so our claim is proven.
%\footnote{The identity (1.7) in \cite{TT} is very useful to see this fact. Note that the terms $H^{*-1}(K_S)$ in that identity carry nontrivial $\C^*$-weights, and the terms $H^*(\O_S)$ are $\C^*$-fixed. The former terms are not removed in our theory but they are removed in \cite{TT}; the latter terms are removed in both our theory and theirs.}.  %When $p_g(S)=0$ (note that in this paper $H^1(\O_S)=0$), the virtual tangent bundle of the moduli space  of Higgs pairs is equivalent in K-theory with  $(\Eb_{\red})^\vee$. This follows by a comparison with \cite[Corollary 3.18]{TT}.

 In particular, the resulting virtual fundamental classes on the fixed loci of both moduli spaces in two papers coincide (because the virtual fundamental class only depends on the K-theory classes of the virtual tangent bundles).

  Tanaka and Thomas define Vafa-Witten invariants $\VW_h(v)\in \mathbb{Q}$ by taking the $\C^*$-equivariant residue of the class of $1$.
They have computed the invariants $\VW_h(v)$ in some interesting examples and express the generating functions of the invariants of certain components of the $\C^*$-fixed locus in terms of algebraic functions.  They were also able to match the invariants  $\VW_h(v)$ (after adding the contributions of all $\C^*$-fixed loci and combining with the calculations in \cite{GK17}) with the few first terms of the modular forms of  \cite{VW94}. Their calculation provides compelling evidence that the invariants $\VW_h(v)$ have modular properties that match with S-duality predictions.
 
By the discussion above about the fixed/moving parts of obstruction theories in this paper and in \cite{TT} we see that if we choose $\alpha=1$ in Definition \ref{DThv4} (see Remark \ref{alpha1}):
$$\DT^{\omega_S}_h(v;1)=\s^{-p_g}\VW_h(v).$$

\section{Description of the fixed locus of moduli space}
We continue this section by giving a precise description of the components of  $\mMb_h(v)^{\C^*}$. Suppose that $\cE$ is a closed point of $\mMb_h(v)^{\C^*}$. Because $\cE$ is a pure $\C^*$-equivariant sheaf, up to tensoring with a power of $\t$, we can assume that, for some partition $\lambda \vdash r$, with $\lambda=(\lambda_1\le\cdots \le \lambda_{\ell(\lambda)})$, we have $$\q_* \cE=\bigoplus_{i=0}^{\ell(\lambda)-1} E_{-i}\otimes \t^{-i},$$ where $E_{-i}$ is a rank $\lambda_{i+1}$ torsion free sheaf on $S$, and the $\O_X$-module structure on $\cE$ is given by a collection of injective maps of $\O_S$-modules (using \eqref{piox}):
$$\psi_i: E_{-i} \to E_{-i-1}\otimes \cL, \quad  \quad i=0,\dots, \ell(\lambda)-1.$$
Let $\cE_i:=z_*E_{-i}$, for any $i$ and let $\cE'_0:=\cE$. Define $\cE'_i$ for $i >0$ inductively by \begin{equation} \label{sescei} 0\to \cE'_{i+1}\otimes \t^{-1}\to \cE'_i \to \cE_i\to 0.\end{equation} Therefore, we get a filtration (forgetting the equivariant structures) $$\cE'_{\ell(\lambda)-1}\subset \cdots\subset \cE'_1\subset\cE'_0=\cE,$$ and the stability of $\cE$ imposes the following conditions:
\begin{equation}\label{stabs} \mu_h(\cE'_i)<\mu_h(\cE) \quad \quad i=1,\dots, \ell(\lambda).\end{equation}
Note that for all $j$ we have $$\q_*\cE'_j\otimes \t^{-j}=\bigoplus_{i=j}^{\ell(\lambda)-1} E_{-i}\otimes \t^{-i},$$ and hence \eqref{stabs} imposes some restrictions on the ranks and degrees of $E_{-i}$'s.

%????{\color{red} perhaps this paragraph not needed}????Conversely, given $\lambda$ and a collection of rank $\lambda_i$ torsion free sheaves $E_{-i}$ on $S$ and the injective maps $\psi_i$ as above, one obtains a unique closed point $\cE$ of $\mMb_h(v)$ by using the equivalence of the categories of $\O_X$ and $\pi_*\O_X
%$-modules {\color{red}???????}. 

This construction also works well for the $B$-points of the moduli space $\mMb_h(v)^{\C^*}$ for any $\C$-schemes $B$. As a result, one gets a decomposition of the $\mathbb{C}^{\ast}$-fixed locus 
$\mMb_h(v)^{\mathbb{C}^{\ast}}$ into connected components
$$\mMb_h(v)^{\C^*}=\coprod_{\lambda \vdash r}\mMb_h(v)^{\C^*}_{\lambda},$$ 
where in the level of the universal families
\begin{align}\label{decompos} &\q_*\left(\eEb|_{X\times \mMb_h(v)^{\C^*}_{\lambda}}\right)=\bigoplus_{i=0}^{\ell(\lambda)-1} \eE_{-i}\otimes \t^{-i},\\ \notag
&\Psi_i: \eE_{-i} \to \eE_{-i-1}\otimes \cL, \quad  \quad i=0,\dots, \ell(\lambda)-1,\\ \notag
& \eEb'_{\ell(\lambda)-1}\subset \cdots\subset \eEb'_1\subset\eEb'_0:=\eEb,
\end{align} 
in which $\eE_{-i}$ is a flat family\footnote{Since $\q$ is an affine morphism, $ \q_*\eEb$ is flat over $\mMb_h(v)^{\C^*}_{\lambda}$, and hence each weight space $\eE_{-i}$ is flat over $\mMb_h(v)^{\C^*}_{\lambda}$.} of
rank $\lambda_i$ torsion free sheaves on $S\times \mMb_h(v)^{\C^*}_{\lambda}$, $\Psi_i$ is a family of fiberwise injective maps over $\mMb_h(v)^{\C^*}_{\lambda}$, $\eEb_i:=z_* \eE_{-i}$,  and $\eEb'_i$ for $i >0$ are inductively defined by \begin{equation}\label{univeses} 0\to \eEb'_{i+1}\otimes \t^{-1}\to \eEb'_i \to \eEb_i\to 0.\end{equation}

%\begin{defi} Let $X$ and $\Eb$ be as defined above. Define
%\begin{align}\label{DThv3}
%\DT_h(v):=
%\int_{[\overline{\mM}_h(v)^{\mathbb{C}^{\ast}}]^{\rm{vir}}}
%\frac{1}{e(E^{\bullet,\mov})},
%\end{align} where $e(-)$ indicates the equivariant Euler class.
% \end{defi}

%to be the contribution of the component $[\mMb_h(v)_\lambda^{\C^*}]^{\vir}$ to  \eqref{DThv3}.

%The $\mathbb{C}^{\ast}$-fixed locus 
%$\overline{\mM}_h(v)^{\mathbb{C}^{\ast}}$
%decomposes into two components
%\begin{align*}
%\overline{\mM}_h(v)^{\mathbb{C}^{\ast}}=
%\mM_h(v) \coprod \widehat{\mM}_h(v)
%\end{align*}
%where $\widehat{\mM}_h(v)$ is the moduli space of 
%$\mathbb{C}^{\ast}$-fixed 
%$\mu_h$-stable sheaves with rank one 
%on the fat surface $rS$. Let

In the rest of the paper, we only study the two extreme cases $\lambda=(r)$ and $\lambda=(1^r)$. By the construction, it is clear that the former case coincides set theoretically with the moduli space $\mM_h(v)$; as we will see in the next section the latter case is related to the nested Hilbert schemes on $S$.  Note that when $r=2$, these cases are the only possibilities, and hence 
\begin{prop} \label{virdec} Suppose that $r=2$, then
$$[\mMb_h(v)^{\C^*}]^{\vir}_{\red}=[\mMb_h(v)^{\C^*}_{(2)}]^{\vir}_{\red}+[\mMb_h(v)^{\C^*}_{(1^2)}]^{\vir}_{\red}$$  \end{prop} \qed 

\subsection{Moduli space of stable torsion free sheaves as fixed locus}
\begin{notn} We sometimes use $-\cdot \t^a$ instead of $-\otimes \t^a$ to make the formulas shorter. We also let $\s=c_1(\t)$.
\end{notn} 
\begin{prop} \label{(r)}
 We have the isomorphism of schemes $\mMb_h(v)_{(r)}^{\C^*}\cong \mM_h(v)$. Moreover, under this identification, we have the following isomorphisms 
\begin{align*}
\Ebf_{\red}|_{\mMb_h(v)^{\C^*}_{(r)}}
&\cong
\big(\dR \hom_{p}(\eE, \eE)_0[1]\big)^{\vee},\\
\Ebm|_{\mMb_h(v)^{\C^*}_{(r)}}
&\cong
\big(\tau^{\le 1}\dR \hom_{p}(\eE, \eE\otimes \cL\cdot \t)\big)^{\vee}.
\end{align*}
In particular,  $\Ebf_{\red}|_{\mMb_h(v)^{\C^*}_{(r)}}$ is identified with the natural perfect trace-free obstruction theory over $\mM_h(v)$, and hence $[\mMb_h(v)^{\C^*}_{(r)}]^{\vir}_{\red}=[\mM_h(v)]^{\vir}_0$.
\end{prop}
\begin{proof} 
The first claim follows by the description above and noting that in this case \eqref{decompos} and \eqref{univeses} give
$$\eEb|_{X\times \mMb_h(v)^{\C^*}_{(r)}}\cong z_*(\eE \boxtimes \N), \quad \quad \q_*\left(\eEb|_{X\times \mMb_h(v)^{\C^*}_{(r)}}\right)=\eE\boxtimes \N,$$ for a line bundle $\N$ on $\mM_h(v)$, by the universal properties of the moduli spaces. For the second part, 
by  \cite[Corollary 11.4]{H06}, we have the following natural exact triangle $$ \eE\boxtimes \N \otimes \cL^{-1}\otimes \t^{-1} [1]\to \dL z^{*}\eEb \to \eE\boxtimes \N,$$ which implies, by adjunction,  the exact triangle $$z_*\,\dR \hom(\eE, \eE)\to \dR \hom(\eEb, \eEb) \to z_*\,\dR \hom(\eE, \eE\otimes \cL\cdot \t)[-1].$$ 
Taking the trace free part, shifting by 1, pushing forward, dualizing, and taking the $\C^*$-fixed part of this exact triangle, we get the first isomorphism; pushing forward, applying the truncation $\tau{\le 1}$, and taking the $\C^*$-moving part of this exact triangle, we get the second isomorphism.
\end{proof}
\begin{cor} \label{int(r)} 
\begin{align*}
\DT_h^{\cL}(v;\alpha)_{(r)}&=\int_{[\mM_h(v)]^{\vir}_0}\frac{\s^\kappa\cdot \alpha}{e\left(\dR \hom_{p}(\eE, \eE\otimes \cL \cdot \t)\right)},\\
\DT_h^{\cL}(v)_{(r)}=\chi^{\vir}(\mM_h(v))&=\int_{[\mM_h(v)]^{\vir}_0}c_d\left(\dR \hom_{p}(\eE, \eE)_0[1]\right)\\&=\int_{[\mM_h(v)]^{\vir}_0}c_d\left(\ext^1_{p}(\eE,\eE)-\ext^2_{p}(\eE, \eE)_0\right),
\end{align*}
where $d$ is the virtual dimension of $\mM_h(v)$, and  $\kappa=1$ if $\cL=\omega_S$, otherwise $\kappa=0$.
\end{cor}
\begin{proof}
To see the first formula, by Proposition \ref{(r)} we can write
\begin{align*}
\DT_h^{\cL}(v;\alpha)_{(r)}&=
\int_{[\mMb_h(v)_{(r)}^{\C^*}]^{\vir}_{\red}}
\frac{\alpha}{e((E^{\bullet,\mov})^\vee)}\\&=
\int_{[\mM_h(v)]^{\vir}_0}
\frac{\alpha}{e\left(\tau^{\le 1}\dR \hom_{p}(\eE, \eE\otimes \cL\cdot \t)\right)}\\&=
\int_{[\mM_h(v)]^{\vir}_0}
\frac{\s^\kappa\cdot \alpha}{e\left(\dR \hom_{p}(\eE, \eE\otimes \cL\cdot \t)\right)}.
\end{align*}
For the last equality, note that the trace map and Grothendieck-Verdier duality induces $$\ext^2_{p}(\eE,\eE\otimes \cL \cdot \t)\cong \hom_{p}(\eE,\eE\otimes \cL^{-1} \otimes \omega_S\cdot \t^{-1})^*\cong \left(p_{ *}(\cL^{-1}\otimes \omega_S \cdot \t^{-1})\right)^*,$$ and by \eqref{antican}, $p_{ *}(\cL^{-1}\otimes \omega_S \cdot \t^{-1})=0$ unless $\cL=\omega_S$ in which case it is $\O_{\mM}\otimes \t^{-1}$.

The second formula in corollary follows directly from Proposition \ref{(r)}, by noting that $\ext^1_{p}(-,-)_0=\ext^1_{p}(-,-)$ by the assumption $H^1(\O_S)=0$, and that $\hom_{p}(\eE,\eE)_0=0$ by the simplicity of the fibers of $\eE$. 
\end{proof}

%\begin{cor}
%$$\DT_h(v)_{(r)}=(-1)^{d+p_g}\s^{p_g}\DT_h(v)_{(r)}.$$
%\end{cor}
%\begin{proof}
%This is true because $\hom_{p_{\mM}}(\eE,\eE)\cong \O_{\mM}$ and $\ext^2_{p_{\mM}}(\eE,\eE)\cong \ext^2_{p_{\mM}}(\eE,\eE)_0\oplus \O_{\mM}^{p_g}$.
%\end{proof}

\begin{cor}
If $\cL=\omega_S$ and $\alpha=1$ then $\DT^{\omega_S}_h(v;1)_{(r)}=(-1)^{d}\s^{-p_g}\DT^{\omega_S}_h(v)_{(r)}$, where $d$ is the virtual dimension of $\mM_h(v)$.
\end{cor}
\begin{proof} By Corollary \ref{int(r)},
\begin{align*}
\DT^{\omega_S}_h(v;1)_{(r)}&=
\int_{[\mM_h(v)]^{\vir}_0}
\frac{\s }{e\left(\dR \hom_{p}(\eE, \eE\otimes \omega_S\cdot \t)\right)}\\&=
\int_{[\mM_h(v)]_0^{\vir}}
\frac{(-1)^{-1+d-p_g} \s}{e\left(\dR \hom_{p}(\eE, \eE\cdot \t^{-1})\right)}\\&=
\int_{[\mM_h(v)]_0^{\vir}}
(-1)^{d}\s^{-p_g}c_d\left(\ext^1_{p}(\eE,\eE)-\ext^2_{p}(\eE, \eE)_0\right),
\end{align*} and then use Corollary \ref{int(r)} again.
Here we used Grothendieck-Verdier duality in the second equality, Lemma \ref{oylah}, and the identities \begin{align*} &e( \left(\hom_{p}(\eE,\eE\cdot \t^{-1})\right)=e(\O_{\mM}\cdot \t^{-1})=-\s,\\ &e\left(\ext^2_{p}(\eE,\eE\cdot \t^{-1})\right)=e\left(\dR^2p_*\O\cdot \t^{-1}\right)\cdot e\left(\ext^2_{p}(\eE,\eE\cdot \t^{-1})_0\right)=(-\s)^{p_g}e\left(\ext^2_{p}(\eE,\eE\cdot \t^{-1})_0\right),\\
&\rank\left[\dR \hom_{p}\left(\eE, \eE\cdot \t^{-1}\right)\right]=1-d+p_g, \end{align*} in the third equality. 
%and the last equality is because of the Grothendieck-Verdier duality. The first formula in proposition now follows because of $c_k(A^{\bullet\, \vee})=(-1)^kc_k(A^\bullet)$ and $c_k(A^{\bullet})=c_k(A^\bullet[-2])$ for any perfect complex $A^\bullet$. 
\end{proof}

\begin{lem} \label{oylah} If $A_\bullet$ is a finite complex of vector bundles and $\t$ is the trivial line bundle with the $\C^*$-action of weight 1 over a scheme  then, for any integer $b$
$$e\big((A_\bullet\cdot \t^b)^\vee\big)=(-1)^{\rank(A_\bullet)} e\big(A_\bullet\cdot \t^{b}\big).$$
\end{lem}
\begin{proof}
In K-theory $A_\bullet$ is equivalent to $A_1-A_2$ where $A_i$ is a vector bundle of rank $a_i$.
\begin{align*}
e\big((A_\bullet\cdot \t^b)^\vee\big)&=\frac{e(A^\vee_1\cdot \t^{-b})}{e(A^\vee_2\cdot \t^{-b})}=\frac{c_{a_1}(A^\vee_1)-b\, \s\cdot c_{a_1-1}(A^\vee_1)+\dots+(-b \,\s)^{a_1}}{c_{a_2}(A^\vee_2)-b\, \s\cdot c_{a_2-1}(A^\vee_2)+\dots+(-b \,\s)^{a_2}}\\&=\frac{(-1)^{a_1}}{(-1)^{a_2}}\cdot \frac{c_{a_1}(A_1)+b\, \s\cdot c_{a_1-1}(A_1)+\dots+(b \,\s)^{a_1}}{c_{a_2}(A_2)+b\, \s\cdot c_{a_2-1}(A_2)+\dots+(b \,\s)^{a_2}}\\&
=(-1)^{\rank(A_\bullet)}\frac{e(A_1\cdot \t^{b})}{e(A_2\cdot \t^{b})}=(-1)^{\rank(A^\bullet)} e\big(A_\bullet\cdot \t^{b}\big).
\end{align*}
\end{proof}

\subsection{Nested Hilbert schemes on $S$}
\subsubsection{Review of the results of \cite{GSY17a}} Let $\S{\n}_\betab$ be the nested Hilbert scheme as in Section \ref{mainresults}. When $r=2$, and so $\n=n_1, n_2,\; \betab=\beta_1$, we have the following well-known special cases:
\begin{enumerate}[1.]
\item $n_2=0,\; \beta_1=0$. The Hilbert scheme of $n_1$ points on $S$, denoted by $\S{n_1}$. It is nonsingular of dimension $2n_1$.
\item $n_1=n_2=0,\; \beta_1\neq 0$. The Hilbert scheme of divisors in class $\beta_1$, denoted by $S_{\beta_1}$. It is nonsingular if $H^{i\ge 1}(L)=0$ for any line bundle $L$ with $c_1(L)=\beta_1$.
\item $n_2=0$. Then $\S{n_1}_{\beta_1}=\S{n_1}\times S_{\beta_1}$. This is the Hilbert scheme of 1-dimensional subschemes $Z\subset S$ such that $[Z]=\beta_1,\quad c_2(I_Z)=n_1.$
\end{enumerate}

%Let $S$ be a nonsingular projective surface over $\C$. We denote the canonical line bundle on $S$ by $\omega_S$ and $K_S:=c_1(\omega_S)$. 
%For any nonnegative integer $m$ and effective curve class $\beta \in H^2(S,\ZZ)$, we denote by $\S{m}_\beta$ the Hilbert scheme of 1-dimensional subschemes $Z\subset S$ such that $$[Z]=\beta,\quad c_2(I_Z)=m.$$ If $\beta=0$ we drop it from the notation and denote by $\S{m}$ the Hilbert scheme of $m$ points on $S$.  Similarly, in the case $m=0$ but $\beta\neq 0$ we drop $m$ from the notation and use $S_\beta$ to denote the Hilbert scheme of curves in class $\beta$. There are natural morphisms $$\div:\S{m}_\beta \to S_\beta, \quad \det:\S{m}_\beta \to \pic(S), \quad \pts:\S{m}_\beta \to \S{m},$$ where $\div$ sends a 1-dimensional subscheme $Z\subset S$ to its underlying divisor on $S$,  $\det(Z):=\O(\div(Z)),$ and $\pts(Z)$ is the 0-dimensional subscheme of $S$ defined by the ideal $I_Z(\div(Z))$ (see \cite{KM77}). From this description it is easy to see that $\S{m}_\beta\cong\S{m}\times S_\beta$. %The morphism $\div$ is smooth of relative dimension $2m$ with fibers isomorphic to $\S{m}$. 
\begin{notn} We will denote the universal ideal sheaves of $\S{m}_{\beta}$, $\S{m}$, and $S_\beta$ respectively by $\i{m}_{-\beta}$, $\i{m}$, and $\I_{-\beta}$, and the corresponding universal subschemes respectively by $\Z^{[m]}_\beta$, $\Z^{[m]}$, and $\Z_\beta$.  We will use the same symbol for the pull backs of $\i{m}$ and $\I_{-\beta}=\O(-\Z_\beta)$ via $\id \times \pts$ and $\id \times \div$ to $S\times \S{m}_\beta$. We will also write $\i{m}_{\beta}$ for $\i{m}\otimes \O(\Z_\beta)$. Using the universal property of the Hilbert scheme, it can be seen that $\i{m}_{-\beta}\cong \i{m}\otimes\O(-\Z_\beta)$, and hence it is consistent with the chosen notation. Let $\pi:S\times \S{m}_\beta\to \S{m}_\beta$ be the projection, we denote the derived functor $\dR\pi_*\dR\hom$ by $\dR\hom_\pi$ and its $i$-th cohomology sheaf by $\ext^i_\pi$.
\end{notn}

The tangent bundle of $\S{m}$ is identified with 
$$T_{\S{m}}\cong \hom_\pi\left(\i{m},\O_{\Z^{[m]}}\right)\cong \dR\hom_{\pi}\left(\i{m},\i{m}\right)_0[1]\cong \ext^1_\pi\left(\i{m},\i{m}\right).$$ %where the index 0 indicates the the trace-free part.

The nested Hilbert scheme is realized as the closed subscheme
\begin{equation}\label{nested inclusion} \iota:\S{\n}_{\betab} \hookrightarrow \S{n_1}_{\beta_1}\times\cdots \times  \S{n_{r-1}}_{\beta_{r-1}}\times  \S{n_r}. \end{equation}
%
% naturally defined by the $r$-tuples $(Z'_1,\dots, Z'_r)$ of subschemes of $S$ such that  $\pts(Z'_i) \subset  Z'_{i-1}$ is a subscheme for any $1<i\le r$. We drop $\betab$ or $\n$ from the notation respectively when $\beta_i=0$ for any $i$ or $n_i= 0, \beta_i\neq 0$ for any $i$. 
%
%Equivalently, $\S{\n}_{\betab}$ is given by the tuples of subschemes $$(Z_1,\dots, Z_r)\in \S{n_1}\times\cdots \times \S{n_r},\quad (C_1,\dots,C_{r-1})\in S_{\beta_1} \times\cdots \times S_{\beta_{r-1}}$$ together with the nonzero maps $\phi_{i}: I_{Z_{i}}\to I_{Z_{i+1}}(C_{i})$ for any  $1\le i< r.$ Note that each $\phi_i$ is necessarily injective.
%
%\begin{rmk} \label{singular}  Note that $\S{\n}$ is the nested Hilbert scheme of 0-dimensional subschemes  $$Z_1\supset Z_2 \supset \cdots \supset Z_r$$ of $S$. $\S{\n}$ is a projective scheme but it is highly singular in general \cite{C98}. It is proven in \cite{C98} that $\S{\n}$ with $n_r<\cdots<n_2<n_1$ is nonsingular only if $r=2$ and $n_1-n_2=1$. 
%\end{rmk}

The inclusions in \eqref{inclusion} in the level of universal ideal sheaves give the universal inclusions
$$\Phi_i: \i{n_i}\to \i{n_{i+1}}_{\beta_i} \quad  \quad 1\le i< r$$ defined over $S\times \S{\n}_{\betab}$. 
\begin{notn} Let $\pr_i$ be the closed immersion \eqref{nested inclusion} followed by the projection to the $i$-th factor, and let $\pi:S\times \S{\n}_\betab\to \S{\n}_\betab$ be the projection. Then we have the fibered square 
\begin{equation} \label{basch} \xymatrix@=1em{S\times \S{\n}_{\betab} \ar@{^{(}->}[r]^-{\iota'} \ar[d]^{\pi} & S \times \S{n_1}_{\beta_1}\times\cdots \times  \S{n_{r-1}}_{\beta_{r-1}}\times  \S{n_r} \ar[d]^{\pi'}
\\
\S{\n}_{\betab} \ar@{^{(}->}[r]^-{\iota}  & \S{n_1}_{\beta_1}\times\cdots \times  \S{n_{r-1}}_{\beta_{r-1}}\times  \S{n_r} }
\end{equation} where $\pi'$ is the projection and $\iota'=\id\times \iota$.
\end{notn}

Applying the functors $\dR\hom_\pi\left (-,\i{n_{i+1}}_{\beta_i}\right)$ and $\dR\hom_\pi\left (\i{n_{i}},-\right)$ to the universal map $\Phi_i$, we get the following morphisms of the derived category $$\dR\hom_\pi\left(\i{n_{i+1}}, \i{n_{i+1}}\right ) \xrightarrow{\Xi_i} \dR \hom_\pi\left(\i{n_i}, \i{n_{i+1}}_{\beta_i}\right),$$ $$\dR\hom_\pi\left(\i{n_{i}}, \i{n_{i}}\right ) \xrightarrow{\Xi'_i} \dR \hom_\pi\left(\i{n_i}, \i{n_{i+1}}_{\beta_i}\right).$$ %By Remark \ref{singular} the projective scheme $\S{\n}_{\betab}$ is highly singular in general. 

The following theorem is one of the main results of \cite{GSY17a}: %proposition that implies Theorem \ref{thm1} is proven in Section \ref{r-step}:
\begin{theorem}[\cite{GSY17a} Theorem 1 and Proposition 2.4] \label{abs-r}
$\S{\n}_\betab$ is equipped with the perfect absolute obstruction theory $\Fb\to \LL_{\S{\n}_\betab}$ whose virtual tangent bundle is given by \begin{align*}&\Fbv=\cone \left(\left[\bigoplus_{i=1}^r \dR\hom_\pi \left(\i{n_i},\i{n_i}\right)\right]_0 \to \bigoplus_{i=1}^{r-1}\dR \hom_\pi \left(\i{n_i}, \i{n_{i+1}}_{\beta_i}\right)\right),\end{align*}
where the map above is naturally induced from all the maps $\Xi_i$ and $\Xi'_i$, and $[-]_0$ means the trace-free part. As a result, $\S{\n}_\betab$ carries a natural virtual fundamental class $$[\S{\n}_\betab]^{\vir} \in A_d( \S{\n}_\betab), \quad \quad  d=n_1+n_r+\frac{1}{2}\sum_{i=1}^{r-1}\beta_i\cdot(\beta_i-K_S),$$ where $K_S$ is the canonical divisor of $S$.
\end{theorem} \qed

\begin{defi} \label{invs} Suppose that $r=2$ and $M\in \pic(S)$. Define the following elements in  $K(\S{n_1,n_2}_\beta)$: $$\sK^{n_1,n_2}_{\beta;M}:=\left[\dR\pi_*M(\Z_\beta) \right]-\left[\dR\hom_\pi(\i{n_1},\i{n_2}_{\beta}\otimes M)\right], \quad \sG_{\beta;M}:=\left[\dR\pi_*M(\Z_\beta)|_{\Z_\beta}\right].$$  %If $\beta=0$ we will instead use the notation $\sK^{[n_1\ge n_2]}_M:=\sK^{n_1,n_2}_{0;M}$. %We use the same notation for the restriction of $E_M$  to $\SD{\n} \subset \SD{n_1}\times \SD{n_2}$ in case $n_1\ge n_2$. Define the following generating series:
%We define the invariant$
%$$\sN_S(n_1,n_2,\beta; M):=\int_{[\S{n_1,n_2}_\beta]^{\vir}}c_{n_1+n_2}(E^{n_1,n_2}_{\beta,M})\cup c_{e}(K_\beta).$$

We also define the twisted tangent bundles in $K(\S{n_i})$ (and will use the same symbols for their pullbacks to $\S{n_1,n_2}$):$$\sT^M_{\S{n_i}}:=\left[\dR\pi_*M \right]-\left[\dR\hom_\pi(\i{n_i},\i{n_i}\otimes M)\right].$$ Note $\sT^{\O_S}_{\S{n_i}}=[T_{\S{n_i}}]$.

Let $\cP:=\cP(M,\beta,n_1,n_2)$ be a polynomial in the Chern classes of $\sK^{n_1,n_2}_{\beta;M}$, $\sG_{\beta;M}$, and $\sT^M_{\S{n_i}}$, then, we can define the invariant
$$\sN_S(n_1,n_2,\beta; \cP):=\int_{[\S{n_1,n_2}_\beta]^{\vir}}\cP.$$
%If the condition \eqref{scond} is satisfied, we can define the reduced invariants $$\sN^{\red}_S(n_1,n_2,\beta; \cP):=\int_{[\S{n_1,n_2}_\beta]^{\vir}}\cP.$$
\end{defi}
\begin{defi}
Let $M\in \pic(S)$. Define an element of $K(\S{n_1}\times \S{n_2})$ as  $$\sE_M^{n_1,n_2}:=\left[\dR\pi'_*M\right]-\left[\dR\hom_{\pi'}(\i{n_1},\i{n_2}\otimes M)\right].$$ If $M=\O_S$ then we will drop it from the notation. \end{defi}

The following results are proven in \cite{GSY17a}:
\begin{theorem}[\cite{GSY17a} Theorem 6] \label{product}
%Suppose the class $\cP\in H^*(\S{n_1,n_2})$ in Definition \ref{invs} is well-behaved under good degenerations of $S$ (in the sense of \cite[Remark 5.10, Proposition 5.11]{GSY17a}), and is the pullback of a class $\cP'$ from $H^*(\S{n_1}\times\S{n_2}_\beta)$ via $\iota^*$. Then 
%$$\sN(n_1,n_2,0;\cP)=\int_{\S{n_1}\times\S{n_2}}c_{n_1+n_2}(\sE^{n_1,n_2}) \cup \cP'.$$
Let $L_1,\dots, L_s$, $L'_1,\dots, L'_{s'}$, $M_1,\dots, M_{t}$, be some line bundles on the nonsingular projective surface $S$, and $l_1,\dots, l_s$, $l'_1,\dots, l'_{s'}$, $m_1,\dots, m_{t}$ be finite sequences of $\pm 1$. Define
$$\cP:=\prod_{i=1}^s c(\sT^{L_i}_{\S{n_1}})^{l_i}\cup\prod_{i=1}^{s'} c(\sT^{L'_i}_{\S{n_2}})^{l'_i}\cup \prod_{i=1}^t c(\kk{n_1, n_2}_{0;M_i})^{m_i}.$$
Then,
\begin{align*}&\sN_S(n_1,n_2,0;\cP)=\\&\int_{\S{n_1}\times \S{n_2}} c_{n_1+n_2}(\sE^{n_1,n_2})\cup \prod_{i=1}^s c(\sT^{L_i}_{\S{n_1}})^{l_i}\cup\prod_{i=1}^{s'} c(\sT^{L'_i}_{\S{n_2}})^{l'_i}\cup \prod_{i=1}^t c(\sE^{n_1, n_2}_{M_i})^{m_i}.\end{align*}

\end{theorem}\qed

\begin{theorem} [\cite{GSY17a} Proposition 2.9] \label{reducedclass}
Suppose that $p_g(S)>0$ and $$|L|\neq \emptyset \quad \& \quad |\omega_S\otimes L^{-1}|=\emptyset$$ for any line bundle $L$ with $c_1(L)=\beta$. Then $[\S{n_1,n_2}_\beta]^{vir}=0$. In particular, in this case $\sN_S(n_1,n_2,\beta; \cP)=0$ for any choice of the class $\cP$.
\end{theorem}\qed

\subsubsection{Nested Hilbert schemes as fixed locus} Suppose that $\cE$ is a closed point of $\mMb_h(v)_{(1^r)}^{\C^*}$. By what we said above,  $\cE$ determines the rank 1 torsion free sheaves $E_{0},\dots, E_{-r+1}$ on $S$ together with the $\O_S$-module injections $\psi_1,\dots,\psi_{r-1}$. Since $S$ is nonsingular, there exist line bundles $L_1,\dots,L_r$ and the ideal sheaves $I_1,\dots,I_r$ of zero dimensional subschemes $Z_1,\dots, Z_r$ such that $E_{-i+1}\cong I_i\otimes L_i$. We can rewrite the maps $\psi_i$ as  $$\phi_i:I_i\to I_{i+1}\otimes M_i,$$ where $M_i:=L_i^{-1}\otimes L_{i+1}\otimes \cL$. The double dual $\phi^{**}_i:\O_S\to M_i$ defines a nonzero section of $M_i$ and hence either $M_i\cong \O_S$ or $|M_i|\neq \emptyset$.

%Note that since $\phi_i\neq 0$ by the construction, $\deg(M_i)\ge 0$, hence there exists a (possibly empty) 1-dimensional subscheme $C_i$ such that $M_i=\O_S(C_i)$.

 Let $$n_i:=c_2(I_i), \quad \beta_i:=c_1(\cL)+c_1(L_{i+1})-c_1(L_{i}),$$ 
 and let $\d(G):=c_1(G)\cdot h$ for any torsion free sheaf $G$ on $S$. 
By construction we have the following two conditions:

\begin{itemize} 
\item By the injectivity of $\phi_i$, $\beta_i$ is an effective curve class, in particular,  $$\d(M_i)>0 \quad \text{ or } \quad \d(M_i)=0\  \ \& \ \ n_{i+1}\le n_i.$$
\item By the stability of $\cE$, using \eqref{stabs}, $$i\sum_{j=i+1}^r\d(L_j)<(r-i)\sum_{j=1}^i\d(L_j)\quad \quad i=1,\dots, r-1.$$
\end{itemize}
\begin{defi} \label{compatible} We say $$\n:=n_{1},n_{2},\dots,n_{r}, \quad \betab:=\beta_1,\dots,\beta_{r-1},$$  are \emph{compatible with the vector $v=(r,\gamma,m)$}, if the above two conditions are satisfied, and moreover, $$\gamma=\sum_{i=1}^r c_1(L_i), \quad \quad m=\sum_{i=1}^r c_1(L_i)^2/2-n_i.$$
\end{defi}
Conversely, given $L_i$ and $I_i$ as above with the numerical invariants $\n$ and $\betab$ compatible with the vector $v$, and the injective maps $\phi_i$, one can recover a unique closed point of $\mMb_h(v)_{(1^r)}^{\C^*}$.  In fact, since $\q$ is an affine morphism, the collection of $E_{-i+1}=L_i\otimes I_i$ and the maps $\phi_i$ determine a pure $\C^*$-equivariant coherent sheaf $\cE$ on $X$ with $\ch(\q_*\cE)=v$ (see \cite[Ex. II.5.17]{H77}). It remains to show that $\cE$ is $\mu_h$-stable. By \cite[Proposition 3.19]{K11}, it suffices to show that $\mu_h(\F)<\mu_h(\cE)$ for any pure $\C^*$-equivariant subsheaf $0\neq \F\subsetneq \cE$.  
Suppose $\rank(\q_*\F)=s$, so this means that $\F \subseteq \cE'_{r-s}$, and hence $$\mu_h(\F)\le \mu_h(\cE'_{r-s}) < \mu_h(\cE),$$ where the first inequality is because $\rank(\q_*\F)=\rank (\q_* \cE'_{r-s})=s$ and the second inequality is because of \eqref{stabs}.
%Pushing forward this inclusion we get $$\pi_* \F=\bigoplus_{i=j}^{k} F_{-i}\otimes \t^{-i}\subset \bigoplus_{i=0}^{r-1} E_{-i}\otimes \t^{-i}=\pi_*\cE,$$ for some $0\le j\le k\le r-1$ and the rank 1 torsion free sheaves $F_{-i}\subseteq E_{-i}$ on $S$. Now clearly, $\d(F_{-i})\le \d(E_{-i})$

%$\lambda$ and a collection of rank $\lambda_i$ torsion free sheaves $E_{-i}$ on $S$ and the injective maps $\psi_i$ as above, one obtains a unique closed point $\cE$ of $\mMb_h(v)$ by using the equivalence of the categories of $\O_X$ and $\pi_*\O_X

%This proves a set theoretical bijection with $\S{\n}_\betab$ in the notation of Section \ref{sec:general}.

\begin{prop} \label{threefoldtwofold}
For any connected component $\T \subset \mMb_h(v)_{(1^r)}^{\C^*}$, there exist $\n$ and $\betab$ compatible with the vector $v$, such that $\T\cong \S{\n}_{\betab}$ as schemes. 

\end{prop}
\begin{proof}
In this case, \eqref{decompos} gives $$\q_*\left(\eEb|_{X\times \mMb_h(v)^{\C^*}_{(1^r)}}\right)=\bigoplus_{i=0}^{r-1} \eE_{-i}\otimes \t^{-i}, \quad \Psi_i: \eE_{-i} \to \eE_{-i-1}\otimes \cL, \quad  \quad i=0,\dots, r-1.$$ where $\eE_{-i}$ is a flat family of rank 1 torsion free sheaves on $S\times \mMb_h(v)^{\C^*}_{(1^r)}$.  By \cite[Lemma 6.14]{K90} the double duals $\eE_{-i}^{**} $ are locally free, and hence for each $i$ we get a morphism from $ \mMb_h(v)^{\C^*}_{(1^r)}$ to $\pic(S)$. But $H^1(\O_S) = 0$ so $\pic(S)$  is a union of discrete reduced points and hence this morphism is constant on connected components of $\mMb_h(v)^{\C^*}_{(1^r)}$. Pulling back a Poincar\'{e} line bundle shows that $\eE_{-i}^{**} $ restricted to a connected component $\T\subset  \mMb_h(v)^{\C^*}_{(1^r)}$ is isomorphic to $L_i\boxtimes \N_i$ for some line bundle $L_i$ on $S$ and $\N_i$ on  $\T$. Therefore, the restriction of $\eE_{-i} \subset \eE_{-i}^{**} $ to $\T$ is of the form\begin{equation} (\I_{\Z_i}\otimes L_i)\boxtimes \N_i\end{equation}for some subscheme $\Z_i \subset S\times \T$, which must be flat over $\T$ by the flatness of $\eE_{-i}$ and the fact that the $\eE_{-i} \subset \eE_{-i}^{**} $ is fiberwise injective (\cite[Lemma 2.14]{HL10}). Let $n_i$ be the fiberwise length of the subscheme $\Z_i $ over $\T$, which is well-defined by the flatness of $\Z_i$. Let  $\beta_i:=c_1(\cL)+c_1(L_{i+1})-c_1(L_{i}).$ 
Define $$\n:=n_{1},n_{2},\dots,n_{r}, \quad \betab:=\beta_1,\dots,\beta_{r-1}.$$ Then, $\n, \betab$ are clearly compatible with the vector $v$. Let $M_i:=L_i^{-1}\otimes L_{i+1}\otimes \cL$. %. $\C_i \subset S\times |M_i|$
Since the maps $$\Psi_i: \big(\I_{\Z_i}\otimes M_i^{-1}\big)\boxtimes \big(\N_i \otimes \N_{i+1} ^{-1}\big)\to  \I_{\Z_{i+1}}$$ are fiberwise injective over $\T$, there exist subschemes $\Z'_i$ flat over $\T$ such that $$\I_{\Z'_i} =\big(\I_{\Z_i}\otimes M_i^{-1}\big)\boxtimes \big(\N_i \otimes \N_{i+1} ^{-1}\big),$$ and the maps $\Psi_i$ induce the injective  maps $$\I_{\Z'_i}\to  \I_{\Z_{i+1}}.$$
Thus, we obtain a classifying morphism $f:\T\to \S{\n}_\betab$.

%$$\Psi_i: \I_{\Z_i}\to  \I_{\Z_{i+1}}\otimes M_{i}\otimes N_{i+1} \otimes N^{-1}_i$$
%$$\Psi_i: \I_{\Z_i}(-\mathcal{C}_i)\otimes N_i \to  \I_{\Z_{i+1}}\otimes N_{i+1} $$
%$$\Psi_i: \I_{\Z_i}(-\mathcal{C}_i)\otimes N_i \otimes N_{i+1} ^{-1}\to  \I_{\Z_{i+1}}$$
%This gives the fiberwise injective maps $$ \I_{\Z_i}(-\mathcal{C}_i) \to  \I_{\Z_{i+1}}$$ and hence we get a morphism from the component to $\S{n}_\betab$.

Conversely, starting with $\S{\n}_\betab$, where $\n, \betab$ are as in the previous paragraph, we have the universal objects $$\Phi_i: \i{n_i}\to \i{n_{i+1}}_{\beta_i} \quad  \quad 1\le i< r$$ over $S\times \S{\n}_\betab$. Taking double dual we get the sections $$\Phi_i^{**}: \O_{S\times \S{\n}_\betab} \to \O_{S\times \S{\n}_\betab}(\Z_{\beta_i}).$$ 
By the same argument as in the previous paragraph, using $H^1(\O_S)=0$, we can find the line bundles $M_1,\dots, M_{r-1}$ on $S$ and $\N'_1,\dots, \N'_{r-1}$ on $\T$ such that 
$\O(\Z_{\beta_i})\cong M_i \boxtimes \N'_i$, where
as before $M_i$ and $\N'_i$ can be written as $$M_i=L_i^{-1}\otimes L_{i+1}\otimes \cL, \quad \N'_i=\N_i^{-1}\otimes \N_{i+1},$$ and hence $\Phi_i$ is equivalent to \begin{equation} \label{dataphi} \Phi_i:\left(\i{n_i}\otimes L_i
\right)\boxtimes \N_i \to \left(\i{n_{i+1}}\otimes L_{i+1}\otimes \cL\right)\boxtimes \N_{i+1}, \quad \quad 1\le i< r. \end{equation}% Define 
%\begin{equation}\label{JJ}\J_{i}:=\left(\i{n_i}\otimes L_i \right)\boxtimes \N_i \quad 1\le i \le r.\end{equation}
By the discussion before the proposition, and the compatibility of $\n, \betab$ with $v$, the maps  \eqref{dataphi} determine a flat family $\cE$ of stable $\C^*$-equivariant sheaves on $X\times \S{\n}_{\betab}$, and hence an $\S{\n}_\betab$-valued point of $\mMb_h(v)_{(1^r)}^{\C^*}$. Thus, we obtain a classifying morphism $g: \S{\n}_\betab \to \mMb_h(v)_{(1^r)}^{\C^*}$ with the image into the component $\T$ (by the choice of $L_i$). One can see by inspection that $f$ and $g$ are inverse of  each other.
\end{proof} 

%\begin{rmk}
%Given $\n, \betab$ compatible with the vector $v$ there could be several components of $\T \subset \mMb_h(v)_{(1^r)}^{\C^*}$ which are isomorphic to $\S{\n}_{\betab}$. This is because at the level of closed pionts in the description that is giveen in Section 3.2.2, different choices of $L_i, L_{i+1}$ could lead to the same $M_i$. However, if $\pic (S)\cong \ZZ$, then, given such $\n, \betab$, then $\S{\n}_{\betab}$ is isomorphic to a unique component of $\mMb_h(v)_{(1^r)}^{\C^*}$.
%\end{rmk}
%

\begin{rmk}
Proposition \ref{threefoldtwofold} in particular shows that if $\n$ and $\betab$ are compatible with a vector $v$ for which $\mMb_h(v)\neq \emptyset$ (for some choice of $\cL$  and $h$), then $\S{\n}_\betab$ is connected. 

\end{rmk}

The following definition is motivated by the proof of Proposition \ref{threefoldtwofold}.
\begin{defi} \label{defJJ} Suppose that $\S{\n}_\betab$ is a component $\T$ of $\mMb_h(v)^{\C^*}$ as in Proposition \ref{threefoldtwofold}. If $\O(\Z_{\beta_i})\cong M_i \boxtimes \N_i$ for $i=1,\dots, r-1$, where $M_i\in \pic(S)$ and $\N_i\in \pic(\S{\n}_\betab)$, then there are line bundles $L_i\in \pic(S)$ (determined by $\T$) such that  $M_i=L_i^{-1}\otimes L_{i+1}\otimes \cL$. Let $\N_0=\O_{\S{\n}_\betab}$ and define \begin{equation}\label{JJ}\J_{i}:=\left(\i{n_i}\otimes L_i \right)\boxtimes \N_{i-1} \quad 1\le i \le r. \end{equation}
By the proof of Proposition \ref{threefoldtwofold}, the maps $\J_i\to \J_{i+1}\otimes \cL$ induced by the universal maps $\Phi_i$ over $\S{n}_\betab$ give rise to a universal family of stable $\C^*$-equivariant sheaves over $X\times \T$. 
 
\end{defi}

\begin{notn}For any coherent sheaves $\F$, $\G$ on $S\times B$ flat over a scheme $B$, and a nonzero integer $a$, we define
$$\langle\F,\G \cdot \t^a \rangle:=e\big(\dR\hom_\pi(\F,\G\cdot\t^a)\big),$$ where $\pi$ is the projection to the second factor of $S\times B$, 
%$$\langle\F,\F \cdot \t^a \rangle_0:=e\big(\dR\hom_\pi(\F,\F\cdot\t^a)_0\big),$$
and $e(-)$ denotes the equivariant Euler class. 
\end{notn}

In the following proposition we compare the restriction  of the $\C^{*}$-fixed complex $\Ebf_{\red}$  to the component $\T\cong \S{\n}_\betab$  with the obstruction theory of $\Fb$ of Theorem \ref{abs-r}. We also find an explicit expression for the moving part of $\Eb_{\red}$ restricted to $\T$.
 \begin{prop} \label{fixedpart} Using the isomorphism in Proposition \ref{threefoldtwofold}, we have $\Ebf_{\red}|_{\T}\cong\Fb$ (of Theorem \ref{abs-r}) in the K-group.  As a result, $$[\mMb_h(v)_{(1^r)}]_{\red}^{\vir}=\sum_{\tiny\begin{array}{c}
 \T \cong \S{\n}_\betab \\\text{is a conn. comp. of}\\  \mMb_h(v)_{(1^r)}^{\C^*}\end{array}}[\S{\n}_{\betab}]^{\vir}.$$  Moreover, \begin{align*}&e\big((\Ebm)^{\vee}|_{\T}\big)=
\frac{\prod_{\tiny\begin{array}{c}1\le i, j \le r \\i\neq j-1\end{array}} \left \langle \J_i\cdot \t^{-i}, \J_j \otimes \cL \cdot \t^{-j+1}\right\rangle}{\s^{\kappa}\prod_{\tiny\begin{array}{c}1\le i, j \le r \\i\neq j\end{array}} \left\langle \J_i\cdot \t^{-i}, \J_j \cdot \t^{-j}\right \rangle},\end{align*} where $\J_i$ are given in \eqref{JJ}, and $\kappa=1$ if $\cL=\omega_S$, otherwise $\kappa=0$.
\end{prop}

\begin{proof} \textbf{Step 1:} \emph{($r=2$, fixed part of the obstruction theory)} We first prove the case $r=2$. %We use the same notation for the pullback of $L_i$ to $X \times \mMb_h(v)_{(1^r)}^{\C^*}$.
 By the proof of Proposition \ref{threefoldtwofold}, the short exact sequence \eqref{univeses} gives %of the universal objects over $X \times \mMb_h(v)_{(1^r)}^{\C^*}$:

\begin{equation}\label{ses}
0\to z_*\J_{2}\otimes \t^{-1} \to \overline{\eE}|_{\T\times X}\to z_*\J_{1} \to 0,
\end{equation} in which  $\J_{i}$ (defined in \eqref{JJ}) carries no $\C^*$-weights. 
Applying $$\dR \hom(\overline{\eE}, -), \quad \dR \hom(-,z_*\,\J_{1}), \quad \dR \hom(-, z_*\, \J_{2}\cdot \t^{-1}) $$ to \eqref{ses}, we get the exact triangles in $D^b(X\times \T)$ filling respectively the middle row,  and the 1st and 2nd  columns of the following commutative diagram:

%\begin{equation}\label{ee} _{p_{\mMb}}
% \dR p_{\overline{M}\ast} \dR \hom(\overline{\eE}, \tilde{J}_{2}\cdot \textbf{t})\to \dR p_{\overline{M}\ast} \dR \hom(\overline{\eE}, \overline{\eE})\to  \dR p_{\overline{M}\ast} \dR \hom(\overline{\eE}, \tilde{J}_{1}).
%\end{equation}
%On other hand, applying $\dR p_{\overline{M}\ast} \dR \hom(-, \tilde{J}_{1})$ and $\dR p_{\overline{M}\ast} \dR \hom(-, \tilde{J}_{2}\cdot \textbf{t})$ to \eqref{ses} respectively and putting the results together with exact triangle \eqref{ee} gives us

\begin{equation}\label{corners}\hspace{-.3in}
\xymatrix@=1em{
\dR \hom(z_*\,\J_{1},z_*\,\J_{1})[-1] \ar[d] \ar[r]&  \dR \hom(z_*\,\J_{1},z_*\, \J_{2}\cdot \t^{-1})  \ar[d]& \\ \dR \hom(\overline{\eE}, z_*\, \J_{1})[-1]\ar[d] \ar[r] &\dR \hom(\overline{\eE},z_*\, \J_{2}\cdot \t^{-1})\ar[r]\ar[d]& \dR \hom(\overline{\eE}, \overline{\eE})\\ \dR \hom(z_*\,\J_{2}\cdot \t^{-1},z_*\, \J_{1})[-1]\ar[r]& \dR \hom(z_*\,\J_{2} \cdot \t^{-1},z_*\, \J_{2}\cdot \t^{-1})&}
\end{equation}
%Now dualizing the digram would reverse the direction of the arrows.
%\begin{equation}\hspace{-.1in}
%\xymatrix@C=5pt@R=25pt{
%\dR p_{\overline{M}\ast} \dR \hom(\tilde{J}_{2}\cdot \textbf{t}, \tilde{J}_{1})^{\vee}[2]  \ar[d]&& \dR p_{\overline{M}\ast} \dR \hom(\tilde{J}_{2}\cdot \textbf{t}, \tilde{J}_{2}\cdot \textbf{t})^{\vee}[2]  \ar[d]\\\  \dR p_{\overline{M}\ast} \dR \hom(\overline{\eE}, \tilde{J}_{1})^{\vee}[2]  \ar[r]\ar[d]&E^{\bullet}_{\overline{\mM}_h(v)}|_{\widehat{\mM}_h(v)}\ar[r]& \dR p_{\overline{M}\ast} \dR \hom(\overline{\eE}, \tilde{J}_{2}\cdot \textbf{t})^{\vee}[2]  \ar[d]\\  \dR p_{\overline{M}\ast} \dR \hom(\tilde{J}_{1}, \tilde{J}_{1})^{\vee}[2] && \dR p_{\overline{M}\ast} \dR \hom(\tilde{J}_{1}, \tilde{J}_{2}\cdot \textbf{t})^{\vee}[2]}
%\end{equation}
%Now we use fiberwise Serre duality and replace the corner terms with their quasi-isomorphic counterparts. In doing so, note that each appearance of $\omega_{\pi}$ will induce a new factor of $\textbf{t}$:

For any coherent sheaf $\F$ on $S$, by  \cite[Corollary 11.4]{H06}, we have the following natural exact triangle $$ \F\otimes \cL^{-1}\cdot \t^{-1} [1]\to \dL z^{*}z_{*}\F\to \F,$$ which for any other sheaf $\G$ on $S$, by adjunction, implies the exact triangle $$z_*\,\dR \hom_{S}(\F, \G)\to \dR \hom_{X}(z_*\,\F, z_*\,\G) \to z_*\,\dR \hom_{S}(\F, \G\otimes \cL\cdot \t)[-1].$$ Using this and taking the $\C^*$-fixed part of the diagram \eqref{corners}, we get the commutative diagram

%For any coherent sheaf $\F$ on $S$, by \cite[Proposition 11.1]{H}, we have $$i_*\,\dL i^{*}\,i_{*}\,\F\cong i_*\, \F\oplus i_*\, (\F\otimes \omega_S^{-1})[1],$$ which by adjunction, implies for any other sheaf $\G$ on $S$, the exact triangle $$i_*\,\dR \hom_{S}(\F, \G)\to \dR \hom_{X}(i_*\,\F, i_*\,\G) \to i_*\,\dR \hom_{S}(i_*\, \F\otimes \omega_S\cdot \t, \G)[-1].$$ Using this and taking the $\C^*$-fixed part of the diagram \eqref{corners}, we get the diagram 

$$
\xymatrix@=1em{
z_* \dR \hom(\J_{1},\J_{1})[-1] \ar[r] \ar[d] & z_*\,\dR \hom(\J_{1}, \J_{2}\otimes \cL)[-1]  \ar[d]& \\ \dR \hom(\overline{\eE}, z_*\, \J_{1})^{\fix}[-1] \ar[d] \ar[r] &\dR \hom(\overline{\eE},z_*\, \J_{2}\cdot \t^{-1})^{\fix}  \ar[r] \ar[d]& \dR \hom(\overline{\eE}, \overline{\eE})^{\fix} \\ 0\ar[r] &z_*\dR \hom(\J_{2}, \J_{2}) &}
$$ in which the middle row and the 1st and 2nd columns are exact triangles. We conclude that \begin{align*}\dR \hom(\overline{\eE}, z_*\, \J_{1})^{\fix}&\cong z_* \dR \hom(\J_{1},\J_{1})\\
\dR \hom(\overline{\eE}, z_*\,\J_{2}\cdot \t^{-1} )^{\fix}&\cong \cone\left(z_*\, \dR \hom(\J_{2},\J_{2})\to z_* \dR \hom(\J_{1},\J_{2}\otimes \cL) \right)[-1].
\end{align*} From this, and noting that the induced map $z_*\dR \hom(\J_{1},\J_{1})[-1]\to z_*\dR \hom(\J_{2},\J_{2})$ in the diagram is zero, we see that 
 \begin{align*}&\dR \hom(\overline{\eE},  \eEb)^{\fix}\cong \cone \left(  z_* \dR \hom(\J_{1},\J_{1})[-1] \to \dR \hom(\overline{\eE}, z_*\,\J_{2}\cdot \t^{-1} )^{\fix}\right)\cong \\
& \cone \Big(  z_* \dR \hom(\J_{1},\J_{1}) \oplus z_*\, \dR \hom(\J_{2},\J_{2})\to z_* \dR \hom(\J_{1},\J_{2}\otimes \cL) \Big)[-1].
\end{align*} 
Taking trace free part, applying $\dR p_{*}$, and shifting by 1, we get
\begin{align*}&(\Ebf_{\red})^\vee|_{\T}\cong \\& \cone \Big( \left[  \dR \hom_{p}(\J_{1},\J_{1}) \oplus \dR \hom_{p}(\J_{2},\J_{2})\right]_0\to \dR \hom_{p}(\J_{1},\J_{2}\otimes \cL) \Big)\cong \\
& \cone \Big( \left[  \dR \hom_\pi(\i{n_1},\i{n_1}) \oplus \dR \hom_\pi(\i{n_2},\i{n_2})\right]_0\to \dR \hom_\pi(\i{n_1},\i{n_2}_\beta)\Big)\cong \Fbv.
\end{align*}
This proves the claim about the fixed part of the obstruction theory when $r=2$.

\textbf{Step 2:}   \emph{($r=2$, moving part of the obstruction theory)} We use diagram \eqref{corners} in Step 1 again, but this time we take the moving parts:

$$
\xymatrix@=1em{
z_*\,\dR \hom(\J_{1}, \J_{2}\cdot \t^{-1})  \ar[d]&&z_* \dR \hom(\J_{1},\J_{1}\otimes \cL\cdot \t) [-1] \ar[d]\\  \dR \hom(\overline{\eE},z_*\, \J_{2}\cdot \t^{-1})^{\mov}  \ar[r] \ar[d]& \dR \hom(\overline{\eE}, \overline{\eE})^{\mov}\ar[r]& \dR \hom(\overline{\eE}, z_*\, \J_{1})^{\mov}\ar[d]\\ z_*\dR \hom(\J_{2}, \J_{2}\otimes \cL\cdot \t) [-1] && \Ab }
$$ in which $$\Ab:=\cone \Big(z_*\,\dR \hom(\J_{2}, \J_{1}\otimes \cL \cdot \t^2)[-2]\to z_*\,\dR \hom(\J_{2}, \J_{1}\cdot \t)\Big),$$ and the middle row and the 1st and 2nd columns are exact triangles. We conclude that \begin{align*}\dR \hom(\overline{\eE}, z_*\, \J_{1})^{\mov}&\cong \cone\big(\Ab\to  z_* \dR \hom(\J_{1},\J_{1}\otimes \cL\cdot \t)\big)[-1]\\
\dR \hom(\overline{\eE}, z_*\,\J_{2}\cdot \t^{-1} )^{\mov}&\cong \cone\left(z_*\, \dR \hom(\J_{2},\J_{2}\otimes \cL\cdot \t)[-2]\to z_* \dR \hom(\J_{1},\J_{2}\cdot \t^{-1}) \right).
\end{align*} Pushing forward, shifting by 1, and taking the equivariant Euler class, we get

\begin{align*}e(\dR\hom_{\bp}(\overline{\eE},\overline{\eE})^{\mov}[1]|_{\T})=\frac{\langle \J_{1},\J_{1}\otimes \cL\cdot  \t\rangle \cdot \langle \J_{2},\J_{2}\otimes \cL\cdot  \t\rangle \cdot \langle \J_{2},\J_{1}\otimes \cL \cdot \t^2 \rangle}{\langle \J_{1},\J_{2}\cdot  \t^{-1}\rangle  \cdot \langle \J_{2},\J_{1}\cdot \t\rangle}. \end{align*}
Also note that $$e(\ext^3_{\bp}(\overline{\eE},\overline{\eE}))=\begin{cases}1 & \cL\neq \omega_S \\ e(\O_{\mMw}\cdot \t)=\s & \cL=\omega_S.\end{cases}$$
This proves the claim about the moving part of the obstruction theory when $r=2$.

\textbf{Step 3:} \emph{($r>2$)}  Again by the proof of Proposition \ref{threefoldtwofold}, the short exact sequence \eqref{univeses} gives 
$$
0\to \eEb_1'|_{\T\times X}\otimes \t^{-1} \to \overline{\eE}|_{\T\times X} \to z_*\,\J_{1}\to 0,$$
% in which $i:S\hookrightarrow X$ is the inclusion of the zero section, and  $\i{n_i}\otimes L_i$ carries no  $\C^*$-weight. Let $\J_i:=\i{n_i}\otimes L_i$.

%
% there exists a suitable vector $v':=(r-1,\gamma',m')$, and a forgetful {\color{red} ??????} morphism $$f:\mMb(v)_{(1^r)}^{\C^*}\to \mMb(v')_{(1)^{r-1}}^{\C^*}.$$ Let $\eEb'$ be the universal sheaf over $X\times \mMb(v')_{(1,\dots,1)}^{\C^*}$. Then, we have a natural short exact sequence $$
%0\to (\id \times f)^*\,\eEb'\cdot \t^{-1} \to \eEb\to i_*\,(\i{n_1}\otimes L_1)\to 0.
%$$
One can then repeat the argument of Step 1 and Step 2, by replacing  $z_*\,\J_{2}$ with $\eEb'|_{\T\times X}$, and use the induction on $r$ to complete the proof of the proposition.

\end{proof}

%In the following proposition, we compute the moving part of the obstruction theory $\Eb$ restricted to the component $\S{\n}_{\betab}$ in Proposition \ref{threefoldtwofold}. %For simplicity we assume $r=2$, and introduce the following notation:

\begin{cor} \label{moving} 
Suppose that $r=2$, $\D=c_1(\cL)$ then
\begin{align*}&e\big((\Ebm)^{\vee}|_{\T}\big)=
\frac{\langle \i{n_1},\i{n_1}(\D)\cdot  \t\rangle \cdot \langle \i{n_2},\i{n_2}(\D) \cdot  \t\rangle \cdot \langle \i{n_2}_\beta,\i{n_1}(2\D) \cdot \t^{2}\rangle}{\s^\kappa\cdot \langle \i{n_1},\i{n_2}_\beta(-\D)\cdot \t^{-1}\rangle\cdot \langle \i{n_2}_\beta,\i{n_1}(\D) \cdot \t\rangle}.\end{align*}\end{cor}% where $v$ is the sum of the ranks of three complexes $$\dR\hom_\pi(\J_{-1},\J_{-1}), \quad \dR\hom_\pi(\J_{-2},\J_{-2}),\quad  \dR\hom_\pi(\J_{-1},\J_{-2}).$$

\begin{proof} 
By Definition \ref{defJJ}, $\i{n_2}_\beta=\i{n_2}\otimes \cL \otimes L_2 \otimes L_1^{-1}\otimes p^*\N_1$. The result then follows from the formula in Proposition \ref{fixedpart} when $r=2$.

\end{proof}

 \begin{cor} \label{movfix} Suppose that $r=2$. Using the notation of Propositions \ref{threefoldtwofold} and \ref{fixedpart}, Corollary \ref{moving} and Definition \ref{invs}, we have 
\begin{align*}\DT^{\cL}_h(v;\alpha)_{(1^2)}=
\sum_{
 \T \cong \S{n_1,n_2}_\beta}\frac{(-1)^{-\D\cdot \beta-K_S\cdot \D/2+3\D^2/2-\kappa}}{2^{\chi(\cL^2)}(-\s)^{\chi(\cL^2)+\chi(\cL)-\chi(\cL^{-1})-\kappa}}\int_{[\S{n_1,n_2}_\beta]^{\vir}}
\alpha \cup \mathcal{Q}_\T.
\end{align*}
\begin{align*}
\DT^{\cL}_h(v)_{(1^2)}&=\sum_{
 \T \cong \S{n_1,n_2}_\beta}\chi^{\vir}(\S{n_1,n_2}_\beta)=
\sum_{
 \T \cong \S{n_1,n_2}_\beta}\sN_S(n_1,n_2,\beta;\cP_\T) ,\end{align*}
  where all sums are over the connected components $\T \cong \S{n_1,n_2}_\beta$ of  $\mMb_h(v)_{(1^2)}^{\C^*}$, and for any such $n_1,n_2$ and $\beta$, 
  $$\mathcal{Q}_\T:=e(\sT^{\cL\cdot \t}_{\S{n_1}})\cdot e(\sT^{\cL\cdot \t}_{\S{n_2}})\cdot\frac { e(\sG_{\beta; \omega_S\otimes \cL^{-1}}\cdot \t^{-1})\cdot e(\sG_{\beta; \cL^{-1}}\cdot \t^{-1})\cdot e(\sK^{n_1,n_2}_{\beta; \omega_S\otimes \cL^{-2}}\cdot \t^{-2})}{  e(\sK^{n_1,n_2}_{\beta;\omega_S\otimes \cL^{-1}}\cdot \t^{-1})  \cdot e(\sK^{n_1,n_2}_{\beta; \cL^{-1}}\cdot \t^{-1})\cdot e(\sG_{\beta; \omega_S\otimes \cL^{-2}}\cdot \t^{-2})},$$ 
  $$\cP_{\T}:=c\left(T_{\S{n_1}}\right)\cup c\left(T_{\S{n_2}}\right)\cup\frac{ c\left(\sG_{\beta;\O_S}\right)}{c\left(\sK^{n_1,n_2}_{\beta;\O_S}\right)}.$$

  %and  $\cP_{\T}$ is the polynomial in the Chern classes of $\sK^{n_1,n_2}_{\beta;\O_S}$, $\sG_{\beta;\O_S}$, and $T_{\S{n_i}}$ given by $$\cP_{\T}:=\frac{c\left(T_{\S{n_1}}\right)\cup c\left(T_{\S{n_2}}\right)\cup c\left(\hom_\pi\left(\i{n_1},\i{n_2}_{\beta}\right)\right)\cup c\left(\ext^1_\pi\left(\i{n_1},\i{n_2}_{\beta}\right)\right)}{c\left(\ext^1_\pi\left(\i{n_1},\i{n_2}_{\beta}\right)\right)}.$$
\end{cor}
\begin{proof}
The formulas are the direct corollary of Propositions \ref{threefoldtwofold} and \ref{fixedpart} and Corollary \ref{moving}. The first formula follows from the following identities: 
\begin{enumerate}[1.]
\item By Grothendieck-Verdier duality and Lemma \ref{oylah}, for any coherent sheaves $\F$, $\G$ on $S\times \S{n_1,n_2}_\beta$ flat over $\S{n_1,n_2}_\beta$ we have
$$\langle\F,\G\cdot \t^a \rangle=(-1)^{v}\langle\G,\F \otimes \omega_S \cdot \t^{-a} \rangle,$$ where $v$ is the rank of the complex $\dR\hom_\pi(\F,\G )$ and $0\neq a\in \ZZ$.

\item For any $0\neq a\in \ZZ$ and  and $M\in\pic(S)$, $$\frac{e(\sK^{n_1,n_2}_{\beta;M}\otimes \t^a)}{e(\sG_{\beta;M}\otimes \t^a)}=\frac{(a\,\s)^{\chi(M)}}{\langle \i{n_1},\i{n_2}_{\beta}\otimes M \cdot \t^{a}\rangle},\quad \quad  \frac{e(\sT^{M\cdot \t}_{\S{n_i}})}{\s^{\chi(M)}}=\frac{1}{\langle \i{n_i},\i{n_i}\otimes M\cdot  \t\rangle}.$$% \quad \rank(\dR\pi_*M)=\chi(\O_S)-c_1(M)\cdot c_1(M^D)/2,$$

\end{enumerate}
For the second formula note that by definition $$ \DT^{\cL}_h(v)_{(1^2)}=
\sum_{\T \cong \S{n_1,n_2}_\beta}\int_{[\S{n_1,n_2}_\beta]^{\vir}}\frac{c\left(\dR\hom_\pi\left(\i{n_1},\i{n_2}_{\beta}\right)\right)}{c\left(\dR\hom_\pi\left(\i{n_1},\i{n_1}\right)\right)\cdot c\left(\dR\hom_\pi\left(\i{n_2},\i{n_2}\right)\right)}.$$ Then we use $$T_{\S{n_i}}\cong \ext^1_\pi\left(\i{n_i},\i{n_i}\right), \quad c\left(\ext^{j\neq 1}_\pi\left(\i{n_i},\i{n_i}\right)\right)=c\left(\ext^{j\neq 1}_\pi\left(\i{n_i},\i{n_i}\right)_0\right)=1.$$%\cong \hom_\pi(\i{n_2}_\beta,\i{n_1}\otimes \omega_S)^\vee=0.$$ %where the isomorphisms are because of the base change  and the vanishing is because $\d(L_1) > \d(L_2)$. 

\end{proof}
 
%\begin{cor} Suppose that $r=2$, and let $N=L_1^{-1}\otimes L_2\otimes \omega_S$. %Using the notation of Propositions \ref{threefoldtwofold}, \ref{fixedpart} and \ref{moving}, we have 
%\begin{align*}
%\DT_h(v)_{(1^2)}&=
%\sum_{\tiny\begin{array}{c}
% \T \cong \S{\n}_\betab \\\text{is a conn. comp. of}\\  \mMb_h(v)_{(1^2)}^{\C^*}\end{array}}\int_{[\S{\n}_\betab]^{\vir}}\frac{c(E^{})\cdot c(\dR\hom_\pi(\J_{-1},\J_{-1}))}{c(\dR\hom_\pi(\J_{-1},\J_{-2}))}
% \end{cor}
%\begin{proof}
%\end{proof}

\subsection{Complete Intersections} Suppose that $S\subset \mathbb{P}^{k+2}$ with $k\ge 1$ is a very general complete intersection of type $(d_1,\dots,d_k)\neq (2), (3), (2,2)$ and $d_i>1$. Let $r=2$, $h=\O_S(1)$, and $\n=n_1,n_2$ and $\beta$ be compatible with the vector $v$ as defined in Definition \ref{compatible}. Then, $\omega_S=\O_S(-k-3+d_1+\cdots+d_k)$ and by the genericity $\pic S=\ZZ$ (see \cite{L21}).
 If $\cL=\O(\ell)$ and $L_i=\O(l_i)$ for $i=1,2$ we must have (by the conditions before Definition \ref{compatible} and \eqref{antican})
$$\ell \ge -k-3+d_1+\cdots+d_k,\quad \quad l_1>l_2,\quad \quad    \ell+l_2\ge l_1,$$ and if $\ell+l_2= l_1$ then we must have $n_1\ge n_2$.
Therefore, we get \begin{equation}\label{ineq} 0<l_1-l_2\le \ell.\end{equation}

%\begin{itemize}
%\item $l_1>l_2$, 
%\item $\ell+l_2\ge l_1$.
%\end{itemize} 

Note that in this case $\beta=c_1(\O_S(\ell+l_2-l_1))$, and that $\beta^D:=K_S-\beta$ is effective if \begin{equation}\label{effcurve} -k-3+d_1+\cdots+d_k-\ell+l_1-l_2 \ge 0.  \end{equation} These observations lead to the following proposition:

\begin{prop} \label{CI}
Suppose that $v=(r=2,\gamma=c_1(\O(2g+1)),m)$. %\footnote{Note that our assumption $\gcd(r,\gamma\cdot h)=1$ is satisfied.}. 
\begin{enumerate} [1.]
\item If $\ell \le 0$ then $\mMb_h(v)^{\C^*}_{(1^2)}=\emptyset$.
\item If $\ell=1$ then $l_1=g+1,\; l_2=g$ and  $$\mMb_h(v)^{\C^*}_{(1^2)}=\coprod_{\tiny \begin{array}{c}n_1+n_2=-m+b \\ n_1\ge n_2\end{array}}\S{n_1\ge n_2},$$ where $b=d_1\cdots d_k\,(g^2+g+1/2)$.
\item Suppose that $d_1+\cdots+d_k\ge k+3$, $\ell>0$ and $l_1, l_2$ are so that $2g+1=l_1+l_2$, condition \eqref{ineq} is satisfied, but  condition \eqref{effcurve} is not satisfied. If $\S{n_1,n_2}_\beta$ is a nonempty component of $\mMb_h(v)^{\C^*}_{(1^2)}$ with $\beta=c_1(\O_S(\ell+l_2-l_1))$ and $n_1+n_2=(l_1^2/2+l_2^2/2)d_1\cdots d_k-m$,  then $[\S{n_1,n_2}_\beta]^{\vir}=0.$
\end{enumerate}
 \end{prop} 
 \begin{proof} Part 1 follows immediately from \eqref{ineq}. Part 2 follows from \eqref{ineq} and Definition \ref{compatible}. Part 3 follows from Theorem \ref{reducedclass}.
 
 \end{proof}

\begin{cor} \label{cor:CI}
In the notation of Proposition \ref{CI} if $\cL=\omega_S$ then
\begin{enumerate} [1.]
\item If $S$ is a Fano complete intersection  or a $K3$ surface i.e. when $d_1+\cdots+ d_k \le k+3$ then $\mMw_h(v)^{\C^*}_{(1^2)}=\emptyset$.
\item If $S$ is isomorphic to one of the following five complete intersection types $$(5), \; (3,3),\; (4,2), \; (3,2,2),\; (2,2,2,2)$$ then $\mMw_h(v)^{\C^*}_{(1^2)}$ is a disjoint union of the nested Hilbert schemes of points as in Proposition \ref{CI} item (2).
\item If $d_1+\cdots+d_k\ge k+4$ and if $\S{n_1,n_2}_\beta$ is a nonempty component of $\mMw_h(v)^{\C^*}_{(1^2)}$, then condition \eqref{effcurve} is always satisfied.
\end{enumerate}
 \end{cor} \qed

%\begin{rmk}
%The stability of $\eE$ imposes further restrictions on $\deg L_i$. Also discuss when all $C_i=\emptyset$.????????????? TBD
%\end{rmk}

 \begin{cor} \label{nocurves}  If $S$ is isomorphic to one of the five types of very general complete intersections in part (2) of Proposition \ref{CI}, then, 
 \begin{enumerate}[1.]
 \item %Suppose that  the conditions of Remark \cite[Remark 5.10]{GSY17a} hold for the class $\alpha$ (e.g. $\alpha=1$), then,  %\footnote{this condition is only needed for the first equality}, then
 \begin{align*}&\DT^{\omega_S}_h(v;1)_{(1^2)}=\sum_{\tiny \begin{array}{c}n_1+n_2=-m+b \\ n_2\le n_1\end{array}}\frac{(-1)^{d_1\cdots d_k-1}}{2^{\chi(\O_S(2))}(-\s)^{\chi(\O_S)-1}} \\&\int_{\S{n_1}\times \S{n_2}}
 c_{n_1+n_2}(\sE^{n_1,n_2}) \cup\frac { e(\sT^{\O_S(1)\cdot \t}_{\S{n_1}})\cdot e(\sT^{\O_S(1)\cdot \t}_{\S{n_2}}) \cdot e(\sE^{n_1,n_2}_{\O_S(-1)}\cdot \t^{-2})}{ e(\sE^{n_1,n_2}\cdot \t^{-1})\cdot e(\sE^{n_1,n_2}_{\O_S(-1)}\cdot \t^{-1})},
\end{align*}
\item 
\begin{align*}&\DT^{\omega_S}_h(v)_{(1^2)}=\\&\sum_{\tiny \begin{array}{c}n_1+n_2=-m+b \\ n_2\le n_1\end{array}}\int_{\S{n_1}\times \S{n_2}}
c_{n_1+n_2}(\sE^{n_1,n_2}) \cup \frac{c(T_{\S{n_1}})\cup c(T_{\S{n_2}})}{c(\sE^{n_1,n_2})},
\end{align*} where $b=d_1\cdots d_k\,(g^2+g+1/2)$.
\end{enumerate}
\end{cor}
\begin{proof}
This follows from Corollary \ref{movfix} and Theorem \ref{product} by noting $\sG_{0;M}=0$.

\end{proof}

\section{Mochizuki's result and proof of Theorem \ref{thm5}}\label{Moch}
In this section, we assume $r=2$, $\gamma\cdot h$ is an odd number, and  that  $p_g(S)>0$, for instance, any generic hyperplane section of a quintic 3-fold
satisfies this assumption.
%Let 
%\begin{align*}
%v=(2, \gamma, m) \in H^0(S) \oplus H^2(S) \oplus H^4(S)
%\end{align*}
%be such that $h \cdot \gamma$ is an odd number. 
%Then any $h$-semistable sheaf $E \in \Coh(S)$
%with $\ch(E)=v$ is $h$-stable. As 
%Let $\mM_{h}(v)$ be the moduli 
%space of $h$-stable sheaves $E \in \Coh(S)$ 
%with $\ch(E)=v$. 
%For simplicity, we assume that 
%there exists a universal sheaf 
%$\eE \in \Coh(S \times \mM_h(v))$. 
%Let
%$p_\mM$ be the projection from $S \times \mM_h(v)$
%to $\mM_h(v)$. We have the 
%decomposition
%\begin{align*}
 %\dR \hom_{ p_{\mM\ast}}(\eE, \eE) =\dR \hom_{p_{\mM\ast}}(\eE, \eE)_0 
%\oplus \dR p_{\mM\ast} \O_{S \times \mM_h(v)}
%\end{align*}
The perfect obstruction theory (see Corollary \ref{int(r)}) $$\left (\dR \hom_{p}(\eE, \eE)_0[1]\right)^{\vee}$$  
%\begin{align*}
%\dR p_{M\ast} 
% \hom_{ p_{\mM\ast}}(\eE, \eE)_0^{\vee} \to \LL_{\mM_{h}(v)}. 
%\end{align*}
%We have the associated
gives the virtual cycle $[\mM_h(v)]^{\vir}$
whose virtual dimension $d$ is 
\begin{align*}
d=\gamma^2 -4m -3\chi(\O_S). 
\end{align*}
Let $\sP(\eE)$ be a polynomial 
in the slant products $\ch_i(\eE)/b$ for 
elements $b \in H^{\ast}(S)$ and $i\in \mathbb{Z}_{\ge 0}$. 
By the wall-crossing argument using the master space, 
Mochizuki describes the invariant
\begin{align*}
\int_{[\mM_h(v)]^{\vir}} \sP(\eE)
\end{align*}
in terms of Seiberg-Witten invariants and 
certain integration over the Hilbert schemes of points on $S$. 
The SW invariants are defined as follows: 
for a curve class $\c \in H^2(S,\ZZ)$, let $L$ be the line 
bundle on $S$ with $c_1(L)=\c$, which is uniquely determined (up to isomorphism) by 
the assumption $H^1(\O_S)=0$. 
Let $S_{\c}$ be the Hilbert scheme of curves in class $\c$ or equivalently the moduli space of non-zero
morphisms $\O_S \to L$, that is isomorphic to 
$\mathbb{P}(H^0( L))$. By the proposition in \cite[Section 3]{GSY17a}, $\dR\pi_*\O_{\Z_c}(\Z_c)$ is the virtual tangent bundle of a perfect obstruction theory $S_{\c}\cong \mathbb{P}(H^0( L))$. Under this identification, it is easy to see that the tangent and the obstruction bundles $\mathsf{T}(\c)$ and 
$\ob(\c)$ naturally sit in the exact sequences on $\mathbb{P}(H^0( L))$:
$$0\to H^0(\O_S)\otimes \O\to H^0(L)\otimes \O(1)
\to \mathsf{T}(\c)\to 0,$$
\begin{align*}
0 \to H^1( L) \otimes \O(1) \to \ob(\c) &\to H^2( \O_S) \otimes
\O\to H^2( L) \otimes \O(1) \to 0. 
\end{align*}
By  \cite[Proposition 5.6]{BF97}, the $[S_{\c}]^{\vir}=e(\ob(\c))\cap[S_\c]$. Since by our assumption $p_g>0$ a simple argument (cf.~\cite[Proposition~6.3.1]{M02}) shows that  the only way that $e(\ob(c))\neq 0$ is that $h^1(L)-h^2(L)<0$ in which case, $\rank(\ob(\c))=\rank(T(\c))$, i.e. the virtual dimension of $S_\c$ is $0$. Then, by a simple calculation
\begin{align*}
\mathrm{SW}(\c) := \int_{[S_\c]^{\vir}} 1=(-1)^{h^0(L)-1}\left( \begin{array}{c}
p_g -1 \\
h^0(L)-1
\end{array}  \right). 
\end{align*}

%The SW invariant
%is computed as (cf.~\cite[Proposition~6.3.1]{M02})
%\begin{align*}
%\mathrm{SW}(\c)=(-1)^{d(\c)}\left( \begin{array}{c}
%p_g -1 \\
%h^0(L)-1
%\end{array}  \right). 
%\end{align*}

%As in Section \ref{sec:general}, let $\S{n}$ be the Hilbert scheme of $n$-points in $S$, and
%for $i=1, 2$, let 
%$\Z_i \subset S \times \S{n_i}$ be the universal 
%subscheme and 
%$\i{n_i} \subset \O_{S \times S^{[n_i]}}$ be 
%its ideal sheaf. 
Consider the decomposition\footnote{Since by assumption $\gamma\cdot h$ is an odd number, the equality $\gamma_1\cdot h = \gamma_2 \cdot h$ never occurs.}
\begin{align*}
&\gamma_1+\gamma_2=\gamma, \ \gamma_i \in H^2(S,Z)\cap H^{1,1}(S),\\
& \gamma_1\cdot h < \gamma_2 \cdot h, 
\end{align*} and let $L_{\gamma_i}$ be the line bundle on $S$ with $c_1(L_{\gamma_i})=\gamma_i$, and define $\i{n_i}_{L_{\gamma_i}}:=\i{n_i}\otimes L_{\gamma_i}$.
Recall that we use the symbol $\pi'$ to denote all the projections $$S \times S^{[n_i]}\to \S{n_i},\quad \quad S \times S^{[n_1]} \times S^{[n_2]}\to \S{n_1}\times \S{n_2}.$$
%to $S^{[n_1]} \times S^{[n_2]}$.

%we denote by $e^{\gamma_i}$ the 
%line bundle on $S$ 
%whose first Chern class equals to $\gamma_i$,

\begin{notn}
Let $\t'$ is the trivial line bundle on $S$ with the $\C^*$-action of weight 1 on the fibers\footnote{Here we use the symbol $\t'$ to distinguish this line bundle from the equivariant trivial line bundle $\t$ defined before with respect to a different $\C^*$-action.}, and let $\s':=c_1(\t')$. We also consider the rank $n$ tautological vector bundle 
on $S^{[n_i]}$, given by 
\begin{align*}
\v{n_i}_{L_{\gamma_i}}:=
\pi'_{\ast} \left(\O_{\z{n_i}} \otimes L_{\gamma_i} \right).
\end{align*}

\end{notn}

%where  
%All the equivariant sheaves in the derived inner Hom's \marginpar{\color{red}$s=c_1(\t)$????}
%are pulled back to $S \times S^{[n_1]} \times S^{[n_2]}$. 

Following Mochizuki, we define

%Following Mochizuki, we define $$\sQ\left(\i{n_1}_{\gamma_1}\cdot \t'^{-1}, \i{n_2}_{\gamma_2}\cdot \t'\right)$$ to be the 
%Euler class of the following $\C^{\ast}$-equivariant 
%virtual vector bundle on $S^{[n_1]} \times S^{[n_2]}$: 
%\begin{align*}
%\left(\i{n_1}_{\gamma_1}\cdot \t'^{-1}, \i{n_2}_{\gamma_2}\cdot \t'\right)=
%-\dR \hom_\pi\left(\i{n_1}_{\gamma_1}\cdot \t'^{-1}, \i{n_2}_{\gamma_2}\cdot \t'\right)
%- \dR \hom_\pi\left(\i{n_2} _{\gamma_2}\cdot \t', \i{n_1} _{\gamma_1}\cdot \t'^{-1}\right).
%%-\dR \hom_\pi(\i{n_1} e^{\gamma_1 -s}, \i{n_2} e^{\gamma_2+s})
%%- \dR \hom_\pi(\i{n_2} e^{\gamma_2+s}, \i{n_1} e^{\gamma_1 -s}). 
%\end{align*} 

\begin{align*} %\label{def:A}
&\sA(\gamma_1, \gamma_2, v; \sP) := \\ \notag
& \sum_{\begin{subarray}{c}
n_1+n_2=\\ \gamma^2/2-m-\gamma_1\cdot \gamma_2 
\end{subarray}}
\int_{S^{[n_1]} \times S^{[n_2]}} \mathrm{Res}_{\s'=0}
\left( \frac{e\left(\v{n_1}_{L_{\gamma_1}}\right)\cdot \sP \left(\i{n_1}_{L_{\gamma_1}}\cdot \t'^{-1} \oplus \i{n_2}_{L_{\gamma_2}}\cdot \t' \right)\cdot e\left(\v{n_2}_{L_{\gamma_2}}\cdot \t'^2\right)}{(2s')^{n_1 + n_2 -p_g}\cdot \sQ\left(\i{n_1}_{L_{\gamma_1}}\cdot \t'^{-1}, \i{n_2}_{L_{\gamma_2}}\cdot \t'\right)}\right). 
\end{align*}
where 
\begin{align*}
&\sQ\left(\i{n_1}_{L_{\gamma_1}}\cdot \t'^{-1}, \i{n_2}_{L_{\gamma_2}}\cdot \t'\right)=\\&e\left(
-\dR \hom_{\pi'}\left(\i{n_1}_{L_{\gamma_1}}\cdot \t'^{-1}, \i{n_2}_{L_{\gamma_2}}\cdot \t'\right)
- \dR \hom_{\pi'}\left(\i{n_2}_{L_{\gamma_2}}\cdot \t', \i{n_1}_{L_{\gamma_1}}\cdot \t'^{-1}\right)\right).
%-\dR \hom_\pi(\i{n_1} e^{\gamma_1 -s}, \i{n_2} e^{\gamma_2+s})
%- \dR \hom_\pi(\i{n_2} e^{\gamma_2+s}, \i{n_1} e^{\gamma_1 -s}). 
\end{align*}

The following result was obtained by Mochizuki: 
\begin{prop}\emph{(Mochizuki~\cite[Theorem~1.4.6]{M02})}\label{thm:Moc}
Assume that $\gamma\cdot h >2K_S \cdot h$ and $\chi(v) :=\int_S v \cdot td_S \ge 1$. 
Then we have the following formula: 
\begin{align*}
\frac{1}{2}\int_{\mM_h(v)}\sP(\eE) =-\sum_{\begin{subarray}{c}
\gamma_1 + \gamma_2 =\gamma \\
\gamma_1\cdot h < \gamma_2\cdot h
\end{subarray}}
\mathrm{SW}(\gamma_1) \cdot 2^{1-\chi(v)} \cdot \sA(\gamma_1, \gamma_2, v;\sP). 
\end{align*}
\end{prop}\qed
\begin{rmk}
The factor $1/2$ in the left hand side of the formula above
comes from the difference between Mochizuki's convention and ours. 
Mochizuki used the moduli stack of oriented stable sheaves, 
which is a $\mu_2$-gerb over our moduli space $\mM_h(v)$. 
\end{rmk}
\begin{rmk} \label{replacev} 
The assumptions $\gamma\cdot h >2K_S \cdot h$ and $\chi(v) \ge 1$ 
are satisfied if we replace $v$ by $v \cdot L_{k h}$ for $k\gg 0$. Note that tensoring with a bundle does not affect the isomorphism class of $\mMb_h(v)$, and hence in particular the DT invariants $\DT^{\cL}_h(v;\alpha), \DT^{\cL}_h(v)$ remain unchanged.  
\end{rmk}

Recall that the $\mathbb{C}^{\ast}$-fixed locus 
$\mMb_h(v)^{\C^*}$
decomposes into components
\begin{align*}
\mMb_h(v)^{\C^*}=
\mMb_h(v)^{\C^*}_{(2)} \coprod \mMb_h(v)^{\C^*}_{(1^2)}, 
\end{align*} and by Proposition \ref{virdec}, 
$$\DT_h(v;\alpha)=\DT_h(v;\alpha)_{(2)}+\DT_h(v;\alpha)_{(1^2)},\quad \DT_h(v)=\DT_h(v)_{(2)}+\DT_h(v)_{(1^2)}.$$
Recall form Corollary \ref{int(r)}

\begin{align*}
\DT^{\cL}_h(v;\alpha)_{(2)}&=\int_{[\mM_h(v)]^{\vir}} \s^\kappa \cdot e\left(-\dR \hom_{p}(\eE, \eE\otimes \cL\cdot t)\right)\cup \alpha,\\
\DT^{\cL}_h(v)_{(2)}&=\int_{[\mM_h(v)]^{\vir}}c_d\left(-\dR \hom_{p}(\eE, \eE)_0\right),
\end{align*}

Suppose that the class $\alpha$ can be written as a polynomial in $\ch_i(\eE)/b$ for $b \in H^{\ast}(S)$. Both $e\left(-\dR \hom_{p}(\eE, \eE\otimes \cL\cdot t)\right)$ and  $c_d\left(-\dR \hom_{p}(\eE, \eE)_0\right)$
can be expanded as polynomials $\sP_{1}$ and $\sP_{2}$ in  slant products
$\ch_i(\eE)/b$ for $b \in H^{\ast}(S)$ and $\s$ by the application Grothendieck-Riemann-Roch theorem and K\"unneth formula (see for example \cite[Prop 2.1]{GK17} for a detailed calculation). We can thus apply  Proposition~\ref{thm:Moc}
to write $\DT_h(v)_{(2)}$
in terms of 
SW invariants and the integration over the Hilbert schemes of points. Therefore, by Corollaries \ref{movfix} and \ref{nocurves} we have

\begin{prop} \label{DT-nested}
Under the assumption of Proposition~\ref{thm:Moc}, we have
the identity
\begin{align*}
\DT^{\cL}_h(v;\alpha)
&=
-\sum_{\begin{subarray}{c}
\gamma_1 + \gamma_2 =\gamma \\
\gamma_1\cdot h < \gamma_2 \cdot h
\end{subarray}} 
\mathrm{SW}(\gamma_1) \cdot 2^{2-\chi(v)} \cdot \sA(\gamma_1, \gamma_2, v; \sP_{1}\cup \alpha)\\&+
\sum_{\begin{subarray}{c}\T \cong \S{n_1,n_2}_\beta \\\text{is a conn. comp. of}\\  \mMb_h(v)_{(1^2)}^{\C^*} \end{subarray}}
\frac{(-1)^{-\D\cdot \beta-K_S\cdot \D/2+3\D^2/2-\kappa}}{2^{\chi(\cL^2)}(-\s)^{\chi(\cL^2)+\chi(\cL)-\chi(\cL^{-1})-\kappa}}\int_{[\S{n_1,n_2}_\beta]^{\vir}}
\alpha \cup \mathcal{Q_\T}. \end{align*}

\begin{align*}
\DT^{\cL}_h(v)
=
-\sum_{\begin{subarray}{c}
\gamma_1 + \gamma_2 =\gamma \\
\gamma_1\cdot h < \gamma_2 \cdot h
\end{subarray}} 
\mathrm{SW}(\gamma_1) \cdot 2^{2-\chi(v)} \cdot \sA(\gamma_1, \gamma_2, v; \sP_{2})+
\!\!\!\sum_{\begin{subarray}{c}\T \cong \S{n_1,n_2}_\beta \\\text{is a conn. comp. of}\\  \mMb_h(v)_{(1^2)}^{\C^*} \end{subarray}}\!\!\!\!\!\!\sN_S(n_1,n_2,\beta;\cP_\T).\end{align*}%\&+\sum_{\begin{subarray}{c}\T \cong \S{n_1,n_2}_\beta \\\text{is a conn. comp. of}\\  \mMb_h(v)_{(1^2)}^{\C^*} \end{subarray}}\int_{[\S{n_1,n_2}_\beta]^{\vir}}\frac{c\left(T_{\S{n_1}}\right)\cup c\left(T_{\S{n_2}}\right)\cup c\left(\hom_\pi\left(\i{n_1},\i{n_2}_{\beta}\right)\right)}{c\left(\ext^1_\pi\left(\i{n_1},\i{n_2}_{\beta}\right)\right)}.
%\end{align*}
%\begin{enumerate}
%\item if 
%$K_S\cdot h\le0$ then 
%\begin{align*}
%\DT^{\omega_S}_h(v;\alpha)
%&=
%-\sum_{\begin{subarray}{c}
%\gamma_1 + \gamma_2 =\gamma \\
%\gamma_1\cdot h < \gamma_2 \cdot h
%\end{subarray}}
%\mathrm{SW}(\gamma_1) \cdot 2^{2-\chi(v)} \cdot \sA(\gamma_1, \gamma_2, v; \sP_{1})\\
%\DT^{\omega_S}_h(v)
%&=
%-\sum_{\begin{subarray}{c}
%\gamma_1 + \gamma_2 =\gamma \\
%\gamma_1\cdot h < \gamma_2 \cdot h
%\end{subarray}}
%\mathrm{SW}(\gamma_1) \cdot 2^{2-\chi(v)} \cdot \sA(\gamma_1, \gamma_2, v; \sP_{2}),\end{align*}
%\item 

In particular, when $\cL=\omega_S$ and  $S$ is isomorphic to one of five types very general complete intersections $(5)\subset \mathbb{P}^3, \; (3,3)\subset \mathbb{P}^4,\; (4,2)\subset \mathbb{P}^4, \; (3,2,2)\subset \mathbb{P}^5,\; (2,2,2,2)\subset \mathbb{P}^6,$ then, %the DT invariants  $\DT^{\cL}_h(v;\alpha), \DT^{\cL}_h(v)$ are completely expressible as sum of explicit integrals over the product of Hilbert scheme of points.
 \begin{align*}&\DT^{\omega_S}_h(v;1)
=
-\sum_{\begin{subarray}{c}
\gamma_1 + \gamma_2 =\gamma \\
\gamma_1\cdot h < \gamma_2 \cdot h
\end{subarray}}
\mathrm{SW}(\gamma_1) \cdot 2^{2-\chi(v)} \cdot \sA(\gamma_1, \gamma_2, v; \sP_{1}\cup \alpha)
+\frac{(-1)^{d_1\cdots d_k-1}}{2^{\chi(\O_S(2))}(-\s)^{\chi(\O_S)-1}} \\&\sum_{\tiny \begin{array}{c}n_1+n_2=-m+b \\ n_2\le n_1\end{array}}\int_{\S{n_1}\times \S{n_2}}
\frac { c_{n_1+n_2}(\sE^{n_1,n_2}) \cup e(\sT^{\O_S(1)\cdot \t}_{\S{n_1}})\cdot e(\sT^{\O_S(1)\cdot \t}_{\S{n_2}}) \cdot e(\sE^{n_1,n_2}_{\O_S(-1)}\cdot \t^{-2})}{ e(\sE^{n_1,n_2}\cdot \t^{-1})\cdot e(\sE^{n_1,n_2}_{\O_S(-1)}\cdot \t^{-1})},
\end{align*}
\begin{align*}
\DT^{\omega_S}_h(v)
=&
-\sum_{\begin{subarray}{c}
\gamma_1 + \gamma_2 =\gamma \\
\gamma_1\cdot h < \gamma_2 \cdot h
\end{subarray}}
\mathrm{SW}(\gamma_1) \cdot 2^{2-\chi(v)} \cdot \sA(\gamma_1, \gamma_2, v; \sP_{2})\\&
+\sum_{\begin{subarray}{c}n_1+n_2=-m+b \\ n_2\le n_1\end{subarray}}\int_{\S{n_1}\times \S{n_2}}
\frac{c_{n_1+n_2}(\sE^{n_1,n_2}) \cup c(T_{\S{n_1}})\cup c(T_{\S{n_2}})}{c(\sE^{n_1,n_2})},
\end{align*} where $\gamma=c_1(\O_S(2g+1))$ and $b=d_1\cdots d_k\, (g^2+g+1/2)$. Finally, when $S$ is a K3 surface and $\cL=\omega_S=\O_S$ then in the above formulas  for $\DT^{\omega_S}_h(v;1)$ and $\DT^{\omega_S}_h(v)$ only the first summations involving $\sA(-)$ will contribute.
\end{prop} \qed
%\begin{rmk}
%By the same argument and using Corollaries \ref{movfix} and \ref{nocurves}, one can prove a result parallel to Proposition \ref{DT-nested} for the invariants $\DT_h(v;\alpha)$.
%\end{rmk}

\noindent {\tt{amingh@math.umd.edu}},\quad 
\noindent {\tt{University of Maryland}} \\
\noindent {\tt{College Park, MD 20742-4015, USA}} \\\\
\noindent{\tt{artan@cmsa.fas.harvard.edu, Center for Mathematical Sciences and\\ Applications, Harvard University, Department of Mathematics, 20 Garden Street, Room 207, Cambridge, MA, 02139}}\\\\
\noindent{\tt{Centre for Quantum Geometry of Moduli Spaces, Aarhus University, Department of Mathematics,
Ny Munkegade 118, building 1530, 319, 8000 Aarhus C, Denmark}}\\\\
\noindent{\tt{National Research University Higher School of Economics, Russian Federation, Laboratory of Mirror Symmetry, NRU HSE, 6 Usacheva str., Moscow, Russia, 119048}}\\\\
\noindent{\tt{yau@math.harvard.edu}}
\noindent{\tt{Department of Mathematics, Harvard University, Cam- bridge, MA 02138, USA }}
\end{document}